\documentclass[10pt,twoside]{siamltex}

\usepackage{verbatim}
\usepackage{enumerate}
\usepackage{stix}

\usepackage[english]{babel}
\usepackage{graphicx,epstopdf,epsfig}
\usepackage{amsfonts,epsfig,fancyhdr,graphics, hyperref,amsmath,amssymb}

\newcommand{\HE}{Name of Handling Editor}
\newcommand{\DoS}{Month/Day/Year}
\newcommand{\DoA}{Month/Day/Year}
\newcommand{\CA}{Name of Corresponding Author}

\newcommand{\Names}{Franz Achleitner, Anton Arnold, and Volker Mehrmann}
\newcommand{\Title}{Hypocoercivity and hypocontractivity concepts for linear dynamical systems}

\renewcommand{\theequation}{\thesection.\arabic{equation}}

\newtheorem{remark}[theorem]{Remark}
\newtheorem{example}[theorem]{Example}

\usepackage{algorithm}
\usepackage{algorithmic}

\makeatletter
\newenvironment{breakablealgorithm}
  {
   \begin{center}
     \refstepcounter{algorithm}
     \hrule height.8pt depth0pt \kern2pt
     \renewcommand{\caption}[2][\relax]{
       {\raggedright\textbf{\fname@algorithm~\thealgorithm} ##2\par}%
       \ifx\relax##1\relax 
         \addcontentsline{loa}{algorithm}{\protect\numberline{\thealgorithm}##2}%
       \else 
         \addcontentsline{loa}{algorithm}{\protect\numberline{\thealgorithm}##1}%
       \fi
       \kern2pt\hrule\kern2pt
     }
  }{
     \kern2pt\hrule\relax
   \end{center}
  }
\makeatother
\usepackage{tikz}

\newcommand{\C}{\mathbb{C}}
\newcommand{\N}{\mathbb{N}}
\newcommand{\R}{\mathbb{R}}

\newcommand{\Cnn}{\C^{n\times n}}
\newcommand{\Cn}{\C^n}

\newcommand{\bigO}{{\mathcal{O}}}
\newcommand{\dd}[1][x]{\,\operatorname{d}\!#1}

\newcommand{\ddt}{\frac{\dd[]}{\dd[t]}}

\newcommand{\mA}{\mathbf{A}}
\newcommand{\mB}{\mathbf{B}}
\newcommand{\mC}{\mathbf{C}}
\newcommand{\mD}{\mathbf{D}}
\newcommand{\mE}{\mathbf{E}}

\newcommand{\mI}{\mathbf{I}}
\newcommand{\mJ}{\mathbf{J}}

\newcommand{\mP}{\mathbf{P}}
\newcommand{\mQ}{\mathbf{Q}}

\newcommand{\mR}{\mathbf{R}}
\newcommand{\mS}{\mathbf{S}}
\newcommand{\mT}{\mathbf{T}}
\newcommand{\mU}{\mathbf{U}}
\newcommand{\mV}{\mathbf{V}}
\newcommand{\mW}{\mathbf{W}}

\newcommand{\tA}{\widetilde \mA}
\newcommand{\tB}{\widetilde \mB}

\newcommand{\tQ}{\widetilde \mQ}
\newcommand{\tU}{\widetilde \mU}

\newcommand{\tR}{\widetilde \mR}

\newcommand{\tSigma}{\widetilde \Sigma}

\newcommand{\hB}{\widehat \mB}
\newcommand{\hD}{\widehat \mD}

\newcommand{\mAH}{\mA_H}
\newcommand{\mAS}{\mA_S}

\newcommand{\mBH}{\mB_H}
\newcommand{\mBS}{\mB_S}
\newcommand{\mBI}{\mB^{-1}}
\newcommand{\mBIH}{(\mB^{-1})_H}

\newcommand{\mCH}{\mC_H}

\newcommand{\mDHC}{m_{dHC}}
\newcommand{\mDSHC}{m_{dSHC}}
\newcommand{\mHC}{m_{HC}}
\newcommand{\mSHC}{m_{SHC}}

\newcommand{\mHCmBI}{\widetilde{m}_{HC}}

\newcommand{\lambdaMax}[1]{\lambda_{\max}^{#1}}
\newcommand{\lambdaMinAH}{\lambda_{\min}^{\mAH}}
\newcommand{\lambdaMaxAH}{\lambda_{\max}^{\mAH}}
\newcommand{\lambdaMinBH}{\lambda_{\min}^{\mBH}}

\newcommand{\sigmaMaxA}{\sigma_{\max}(\mAd)} 

\newcommand{\sigmaMax}[1]{\sigma_{\max}(#1)}

\newcommand{\mAc}{\mathbf{A}_c} 

\newcommand{\mAd}{\mathbf{A}_d} 
\newcommand{\tAd}{\tA_d} 

\newcommand{\mPd}{\mathbf{P}_d}
\newcommand{\mQd}{\mathbf{Q}_d}
\newcommand{\mUd}{\mathbf{U}_d}

\newcommand{\topH}{{\mathsf{H}}} %
\newcommand{\ip}[2]{\langle {#1}\ ,\, {#2} \rangle}

\DeclareMathOperator{\kernel}{ker}

\newcommand{\bx}{x}

\newcommand{\bv}{v} 

\usepackage{hyperref}

\begin{document}

\bibliographystyle{abbrv}

\setcounter{page}{1}

\thispagestyle{empty}

 \title{\Title\thanks{Received
 by the editors on \DoS.
 Accepted for publication on \DoA.
 Handling Editor: \HE. Corresponding Author: \CA}}

\author{Franz Achleitner\thanks{Technische Universit\"at Wien, Institute of Analysis and Scientific Computing, Wiedner Hauptstra\ss{}e 8-10, A-1040 Wien, Austria, franz.achleitner@tuwien.ac.at},
\and
Anton Arnold\thanks{Technische Universit\"at Wien, Institute of Analysis and Scientific Computing, Wiedner Hauptstr\ss{}e 8-10, A-1040 Wien, Austria, anton.arnold@tuwien.ac.at},
\and
Volker Mehrmann\thanks{Technische Universit\"at Berlin, Institut f.~Mathematik, MA 4-5, Stra\ss{}e des 17.~Juni 136, D-10623 Berlin, mehrmann@math.tu-berlin.de}
} 


\markboth{\Names}{\Title}

\maketitle

\begin{abstract}
For linear dynamical systems (in continuous-time and discrete-time) we revisit and extend the concepts of hypocoercivity and hypocontractivity and give a detailed analysis of the relations of these concepts to (asymptotic) stability, as well as (semi-)dissipativity and (semi-)\-contractivity, respectively.
On the basis of these results, the short-time behavior of the propagator norm for  linear continuous-time and discrete-time systems is characterized by the (shifted) hypocoercivity index and the (scaled) hypocontractivity index, respectively.
\end{abstract}





\section{Introduction}
\label{sec:introduction}

In this paper we discuss different concepts that characterize the short and long time behavior of linear continuous-time ordinary differential equations (ODEs)
\begin{equation}\label{ODE:B}
  \bx'(t) =\mAc \bx(t) =-\mB \bx(t)\,,\quad
  \bx(0)=\bx^0, \quad
  t\geq 0,
\end{equation}
and discrete-time difference equations (DDEs)
\begin{equation}\label{DDE:B}
  \bx_{k+1} = \mAd\bx_{k}\,,\quad
  \bx_0=\bx^0, \quad
  k\in\N_0,
\end{equation}
with matrices $\mAc,\mAd\in\Cnn$.

It is well-known that the long-time behavior of solutions of \eqref{ODE:B} and \eqref{DDE:B} can be characterized via the spectral properties of the matrices $\mAc,\mAd$ or the solutions of Lyapunov equations \cite{Adr95,HiPr10,LanT85,LaS86}.
To understand the short-time behavior of continuous-time systems much progress has recently been made for systems with a semi-dissipative structure, i.e. systems where $\mA_c$ has a semidefinite symmetric part. For this subclass it has recently been observed in \cite{AAC22,AAM21} that the short- and long-time behavior  can be characterized via the concept of hypocoercivity and the hypocoercivity index. For this subclass also the analysis of the long-time behavior becomes simpler and more elegant.

In this paper we show that a similar concept of \emph{hypocontractivity and a hypocontractivity index} is analogously available in the discrete-time case and that it can be characterized via the polar decomposition of $\mAd$.

For both, the continuous- and discrete-time we present a systematic review and analysis of the different concepts and show the subtle differences and similarities to the classical spectral concepts and illustrate these with numerous examples.
Furthermore, we present the close relationship of these concepts to classical controllability and observability concepts in control theory.

Note that we switch in the discussion of~\eqref{ODE:B} between the classical notation with $\mAc$ as is common in dynamical systems and the notation with $-\mB$ as is common in evolution equations.

In Section~\ref{sec:recapctsys} we recall the concepts of (asymptotic) stability, (semi-)dissipativity, and hypocoercivity for con\-tin\-u\-ous-time systems that have been discussed in~\cite{AAM21}.
To better understand the decay behavior of solutions we extend the concept of hypocoercivity to \emph{shifted hypocoercivity}.
We also show under which linear transformations of the system these properties stay invariant.

In the second part of the paper, in Section~\ref{sec:Stability+ODEs:discrete-time} we derive the corresponding results for discrete-time systems and, in particular, analyze the relation between (asymptotic) stability, (semi-)contractivity, and hypocontractivity as well as \emph{scaled hypocontractivity}.

The third part in Section~\ref{sec:dtct} studies how the discussed properties are related under Cayley transformations that map between continuous-time and discrete-time systems. We show that many  properties including the hypocoercivity index and hypocontractivity index map appropriately. However, in general, the shifted hypocoercivity and scaled hypocontractivity indices are not mapped into each other.
Computationally feasible staircase forms to check hypocoercivity for accretive matrices and hypocontractivity for semi-contractive matrices, and to determine the associated indices are discussed in the Appendix.

We use the following notation:
The conjugate transpose of a matrix $\mC\in\Cnn$ is denoted by $\mC^\topH$.
Positive definiteness (semi-definiteness) of a Hermitian matrix $\mC$ 
is denoted by $\mC>0$ ($\mC\geq 0$).

\section{Stability, semi-dissipativity, and hypocoercivity for continuous-time systems}\label{sec:recapctsys}
In this section we recall some properties of linear continuous-time systems and their relationship.
Let us give a simplified definition of stability, for the general definition see e.g. \cite{Adr95,HiPr10}.

\begin{definition}\label{def:stable}
The trivial solution $x\equiv 0$ of \eqref{ODE:B} is called \emph{(Lyapunov) stable} if all solutions of~\eqref{ODE:B} are bounded for $t\geq 0$, and it is called \emph{asymptotically stable} if it is stable and all solutions of~\eqref{ODE:B} converge to $0$ for $t\to \infty$.
\end{definition}

For linear systems~\eqref{ODE:B} a solution is (asymptotically) stable if and only if the trivial solution $x\equiv 0$ is (asymptotically) stable. Therefore, if the trivial solution $x\equiv 0$ of~\eqref{ODE:B} is (asymptotically) stable, then we call the system~\eqref{ODE:B} \emph{(asymptotically) stable}.
%

It is well-known, see e.g.\ \cite{Adr95,HiPr10}, that~\eqref{ODE:B} is \emph{(Lyapunov) stable} if all eigenvalues of~$\mAc$ have non-positive real part and the eigenvalues on the imaginary axis are semi-simple, and it is \emph{asymptotically stable} if all eigenvalues of~$\mAc$ have negative real part.

A  concept closely related to stability is that of (semi-)dissipativity.
Writing $\mAc$ as the sum of its Hermitian part $\mAH :=(\mAc +\mAc^\topH)/2$ and skew-Hermitian part $\mAS :=(\mAc -\mAc^\topH)/2$, we have the following definition, \cite[Definition 4.1.1]{Be18}.
\begin{definition} \label{def:semiDissipative}
A matrix~$\mAc\in\Cnn$ is called~\emph{dissipative} (resp.~\emph{semi-dissipative}) if the Hermitian part~$\mAH$ is negative definite (resp.~negative semi-definite).
For a (semi-)dissipative matrix~$\mAc\in\Cnn$, the associated ODE~\eqref{ODE:B} is called \emph{(semi-)dissipative Hamiltonian ODE}.
Alternatively, 
a matrix~$\mB=-\mAc\in\Cnn$ is called \emph{accretive} (or \emph{positive semi-dissipative}) if its Hermitian part~$\mBH$ is positive semi-definite.
\end{definition}

An nice property of a
 semi-dissipative Hamiltonian ODE~\eqref{ODE:B} is that it is (Lyapunov) stable, since for all solutions of~\eqref{ODE:B} we have
%
\begin{equation} \label{energy:estimate}
\ddt \|x(t)\|^2
= \ip{\mAc x(t)}{x(t)} +\ip{x(t)}{\mAc x(t)}
= \ip{x(t)}{(\mAc^\topH +\mAc)x(t)}
\leq 0 ,
\quad t\geq 0 ,
\end{equation}
i.e.\ the Euclidean norm (which may serve as a \emph{Lyapunov function}), is non-increasing.

The converse is in general not true,
because the Hermitian part of a matrix~$\mAc$ associated with a stable system~\eqref{ODE:B} does not have to be negative semi-definite, as the following example shows:
\begin{example}\label{ex:non-coerciveB}{\rm
Consider the matrix
\begin{equation*} 
 \mB = \begin{bmatrix} 3 & 3 \\ -3 & -1 \end{bmatrix}
\end{equation*}
so that~$\mAc=-\mB$ has eigenvalues $\lambda=-1\pm i\sqrt5$, but the Hermitian part~$\mAH$ is indefinite with eigenvalues $\lambdaMinAH=1$ and $\lambdaMaxAH=-3$.
Hence, the norm of solutions of~\eqref{ODE:B} may increase initially at the rate $e^t$.
}
\end{example}

\begin{remark}[Logarithmic Norm] \label{remark.logarithmic.norm}
Since the flow generated by~\eqref{ODE:B} is given by the matrix exponential~$e^{\mAc t}$, the long-time behavior of the propagator norm $\|e^{\mAc t}\|$, or to be precise---its exponential rate---is determined by the \emph{spectral abscissa}
\begin{equation} \label{1.spectral.abscissa}
 \alpha(\mAc)
:= \max\{\Re(\lambda)\ |\ \text{$\lambda$ is an eigenvalue of $\mAc$}\} \,,
\end{equation}
see e.g.~\cite{vL77}.
\newline\indent
In contrast, the exponential rate of the short-time behavior of~$\|e^{\mAc t}\|$ is determined by the logarithmic norm:
The \emph{logarithmic norm} of a matrix $\mAc\in\C^{n\times n}$ with respect to an inner product is defined as
\begin{equation} \label{1.logarithmic.norm}
 \mu(\mAc)
:= \sup_{\|x\|=1} \Re (\ip{x}{\mAc x} )
 = \max_{\|x\|=1} \Re (\ip{x}{\mAc x}) \ ,
\end{equation}
i.e. $\mu(\mAc)$ is the maximal real part of the \emph{numerical range of~$\mAc$}.
Thus, the solutions~$x(t)$ of~\eqref{ODE:B} satisfy
$\ddt \|x(t)\|^2 = \ip{x}{(\mAc^\topH+\mAc)x}
\leq 2\mu(\mAc)\ \|x(t)\|^2$,
which implies  that
\begin{equation} \label{ODE:short-t} \|x(t)\| \leq e^{\mu(\mAc)\ t}\|x^0\| \quad \text{for } t\geq 0 \ .
\end{equation}
In particular, a matrix~$\mAc\in\Cnn$ is semi-dissipative if and only if $\mu(\mAc)\leq 0$.
\end{remark}

A third related concept is that of hypocoercivity for matrices and the associated hypocoercivity index, which was introduced originally in the context of linear operators see~\cite{AAC18,ArEr14,Vi09}.
\begin{definition}[{Definition 2.5 of~\cite{AAC22}}] \label{def:matrix:hypocoercive}
A matrix $\mC\in\Cnn$ is called~\emph{coercive} (or \emph{strictly accretive}) if its Hermitian part $\mCH$ is positive definite, and it is called~\emph{hy\-po\-coercive} if the spectrum of~$\mC$ lies in the \emph{open} right half plane.
A matrix~$\mAc\in\Cnn$ is called \emph{negative hypocoercive} if the spectrum of~$\mAc$ lies in the \emph{open} left half plane.
\end{definition}
%

The relationship between positive semi-dissipativity and hypocoercivity is characterized by the following result.
%
\begin{proposition}[{\cite[Lemma 3.1]{MMS16}, \cite[Lemma 2.4 with Proposition~1(B2), (B4)]{AAC18}}] \label{prop:border}
Let $\mB\in\Cnn$ be (positive) semi-dissipative.
Then, $\mB$ has an eigenvalue on the imaginary axis if and only if $\mBH v =0$ for some eigenvector~$v$ of~$\mBS$. 
\end{proposition}

Note that, due to the assumptions, purely imaginary eigenvalues of semi-dissipative matrices are necessarily semi-simple, see also~\cite{MehMW18,MehMW20}.
Therefore, an accretive matrix~$\mB$ is hypocoercive if and only if no eigenvector of the skew-Hermitian part lies in the kernel of the Hermitian part.
The latter condition is well known in control theory, and equivalent to the following statements:
\begin{lemma} \label{lem:HC:equivalence}
Let $\mB\in\Cnn$ be accretive.
Then the following are equivalent:
\begin{itemize}
\item [(B1)] 
There exists $m\in\N_0$ such that
\begin{equation}\label{condition:KalmanRank:BS_BH}
 \rank[{\mBH},\mBS{\mBH},\ldots,(\mBS)^m {\mBH}]=n \,.
\end{equation}
\item [(B2)] 
There exists $m\in\N_0$ such that
\begin{equation}\label{Tm:BS_BH}
 \mT_m :=\sum_{j=0}^m \mBS^j \mBH ((\mBS)^\topH)^j > 0 \,.
\end{equation}
%
\item [(B3)] 
No eigenvector of~$\mBS$ lies in the kernel of~$\mBH$.
\item [(B4)] 
$\rank [\lambda \mI-\mBS, \mBH] =n$ for every $\lambda \in \C$ , in particular for every eigenvalue~$\lambda$ of~$\mBS$.
\end{itemize}
Moreover, the smallest possible~$m\in\N_0$ in (B1) 
and (B2) 
coincide.
\end{lemma}
\begin{proof}
The equivalence of (B1), (B3), and (B4) and its proof are classical, see e.g. \cite[Theorem~6.2.1]{Da04} for real matrices, but its proof extends verbatim to complex matrices; see also \cite[Proposition~1]{AAC18}.
The equivalence of (B1) and (B2) follows from Lemma \ref{lem:Definiteness} in the Appendix, setting $\mD:=\mB_H$ and $\mC:=\mB_S$.
\end{proof}

\begin{remark}\label{rem:fullB}
In Lemma~\ref{lem:HC:equivalence} we could have alternatively stated the equivalence of the following conditions, that are equivalent to the corresponding ones in Lemma~\ref{lem:HC:equivalence}.
\begin{itemize}
\item [(B1')] 
There exists $m\in\N_0$ such that
\[ 
 \rank[{\mBH},\mB{\mBH},\ldots,\mB^m {\mBH}]=n \,.
\]
\item [(B2')] 
There exists $m\in\N_0$ such that
\[
\sum_{j=0}^m \mB^j \mBH (\mB^\topH)^j > 0 \,.
\]
\item [(B2'')] 
There exists $m\in\N_0$ such that
\begin{equation}\label{Tm:BS_BH2}
\sum_{j=0}^m (\mB^\topH)^j \mBH \mB^j > 0 \,.
\end{equation}
\item [(B3')] 
No eigenvector of $\mB$ lies in the kernel of $\mBH$.
\item [(B4')] 
$\rank [\lambda \mI-\mB, \mBH] =n$ for every  $\lambda \in \C$ , in particular for every eigenvalue~$\lambda$ of~$\mB$.
\end{itemize}

This is easily seen, since every eigenvector of $\mB$ that is in the kernel of $\mBH$ is immediately an eigenvector of $\mBS$; and conversely, every eigenvector of $\mBS$ that is in the kernel of $\mBH$ is also an eigenvector of $\mB$, see \cite{MehMW18}.
It also follows directly from the staircase forms presented in \cite{AAM21}.
\end{remark}
%

\begin{remark}
The equivalence of properties stated in Proposition~\ref{prop:border}, Lemma~\ref{lem:HC:equivalence} and Remark~\ref{rem:fullB} show that e.g. also the coercivity of the associated matrix~$\mT_m$ in~\eqref{Tm:BS_BH} could have been used to define hypocoercivity for accretive matrices (in the finite-dimensional setting).
Only future research of bounded and unbounded accretive operators on infinite-dimensional Hilbert spaces will decide which is the appropriate characterization for accretive operators to be hypocoercive i.e. to generate a uniformly exponentially stable $C_0$-semigroup.
\end{remark}
\begin{definition}[{\cite[Definition 3.1]{AAC22}}] \label{def:HCI}
Suppose that $\mB\in\Cnn$ is accretive and hypocoercive.
The~\emph{hy\-po\-co\-er\-ci\-vi\-ty index (HC-index)~$m_{HC}$ of the matrix~$\mB$} is defined as the smallest integer~$m\in\N_0$ such that~\eqref{Tm:BS_BH} holds.
\end{definition}

Note that for $\mB\in\Cnn$ (by the Cayley-Hamilton theorem applied to (B1')) it follows immediately that the hypocoercivity index (if it exists) is bounded by $n-1$. More precisely, for a finite hypocoercivity index we even have $\mHC\le \dim \ker (\mBH)\le n-1$ (see Remark 4(b) in \cite{AAC18}). 
Furthermore, a hypocoercive matrix~$\mB$ is coercive if and only if $m_{HC}=0$.

\begin{remark}
Hypocoercive matrices are often called~\emph{positively stable}, whereas negative hypocoercive matrices are often called~\emph{stable}.
Note also that in~\cite[Definition 3]{AAM21}, the HC-index for a semi-dissipative matrix~$\mAc\in\Cnn$ is defined as the HC-index of its accretive counterpart~$\mB=-\mAc$.
We do not make use of this convention here.
\end{remark}

%
Phenomenologically, the HC-index of an accretive matrix~$\mB$ describes the structural complexity of the intertwining of the Hermitian part~$\mBH$ and skew-Hermitian part~$\mBS$ (see~\cite{AAC18} for illustrating examples).
Moreover, for a semi-dissipative Hamiltonian ODE~\eqref{ODE:B}, the HC-index characterizes the short-time decay of the spectral norm of the \emph{propagator} of the associated semigroup $S(t):=e^{-\mB t}\in \Cnn$, $t\geq 0$.
%

\begin{proposition}[{\cite[Theorem 2.7]{AAC22}}]\label{prop:ODE-short}
Let the ODE system~\eqref{ODE:B} be semi-dissipative Hamiltonian with 
(accretive) matrix $\mB\in\Cnn$.
\begin{enumerate}[(a)]
\item \label{th:HC-decay:a}
The (accretive) matrix~$\mB$ is hypocoercive (with hypocoercivity index $\mHC\in\N_0$)
if and only if
\begin{equation}\label{short-t-decay}
  \|e^{-\mB t}\|_2 = 1-ct^a+\bigO(t^{a+1})\quad\text{ for } t\in[0,\epsilon),
\end{equation}
for some $a,c,\epsilon>0$. In this case, necessarily $a= 2m_{HC}+1$.
%
\item \label{th:HC-decay:b}
Consider the ODE \eqref{ODE:B} with $\epsilon$-dependent system matrix $B=\epsilon A +C$ where $\epsilon\in\R$.
If $B=\epsilon A+C$ is hypocoercive for $\epsilon\ne 0$, then the coefficient $c=c_\epsilon$ in the Taylor expansion of the propagator norm~\eqref{short-t-decay} satisfies
\begin{equation} \label{th:HC-decay:epsilon}
0
<\tilde{c}_2 \,\epsilon^{2m_{HC}}
\leq c=c_\epsilon
\leq \tilde{c}_1 \,\epsilon^{2m_{HC}},
\end{equation}
for some positive constants $\tilde{c}_1, \tilde{c}_2>0$ independent of $\epsilon\ne 0$.
\end{enumerate}
\end{proposition}
\begin{remark}
\begin{itemize}
 \item
For genuine semi-dissipative Hamiltonian ODE systems~\eqref{ODE:B} (such that $\mu(\mAc)=0$), the estimate~\eqref{ODE:short-t} based on the logarithmic norm~$\mu(\mAc)$ yields only $\|x(t)\|\leq\|x^0\|$ for $t\geq 0$.
 \item
For semi-dissipative Hamiltonian ODE systems~\eqref{ODE:B}, (a lower bound for) the characterization of the HC-index via the short-time behavior of the propagator norm in~\eqref{short-t-decay} may also be derived by considering a suitable energy-preserving system, see e.g. \cite{Sta05}.
However, the proof of Proposition~\ref{prop:ODE-short} in~\cite{AAC22} yields \emph{quantitative} lower and upper bounds for the multiplicative constant~$c$ in~\eqref{short-t-decay}. These explicit bounds allow to conclude the structural result in Proposition~\ref{prop:ODE-short}~\ref{th:HC-decay:b}.
\end{itemize}
\end{remark}

In Figure~\ref{fig:VennDiagram} we illustrate the relationship between the different concepts that we have discussed so far.
\begin{figure}
\includegraphics[width=\textwidth]{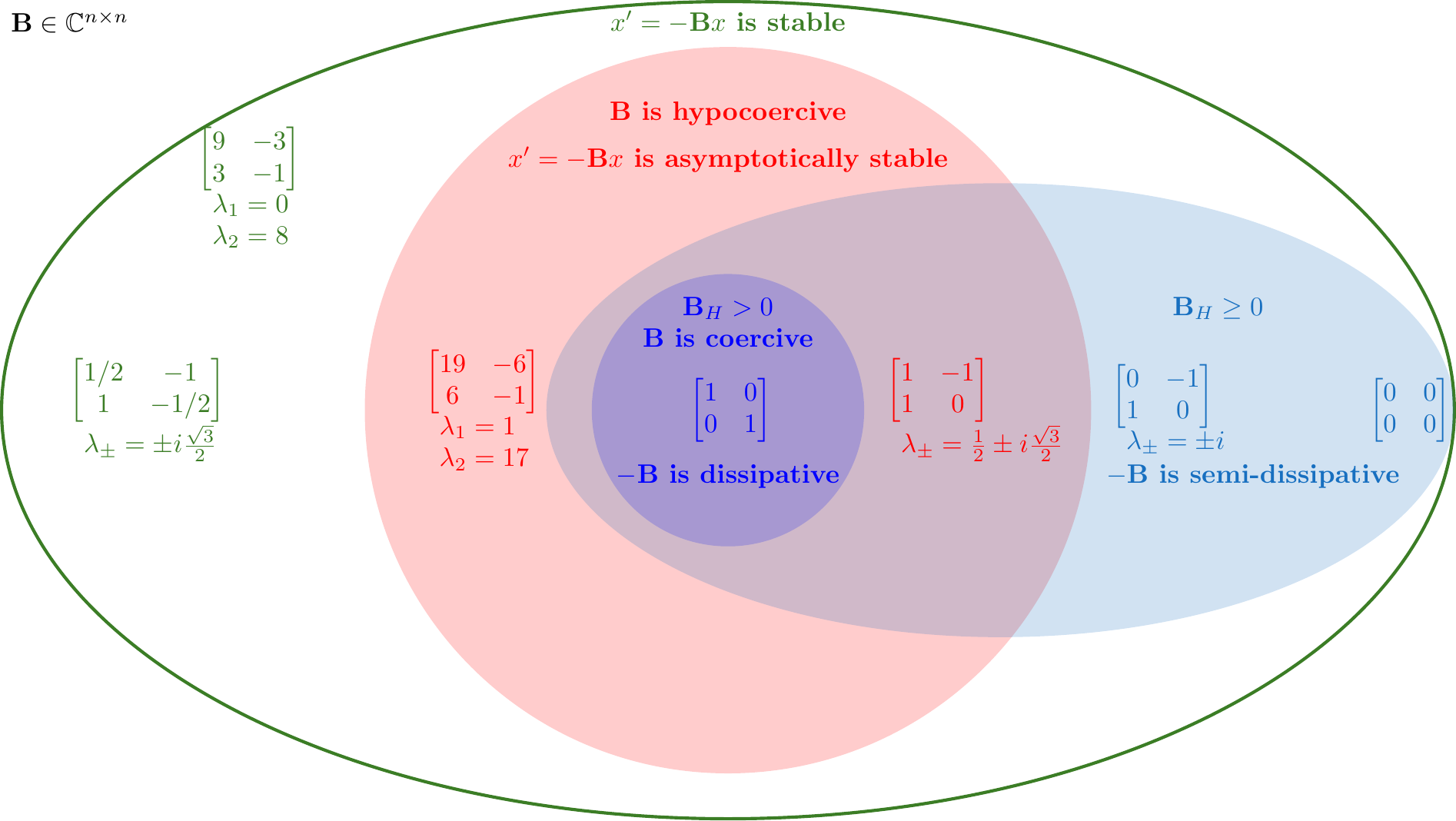}
\caption{Illustration of the relationship between sets of matrices $\mB\in\Cnn$ which are (hypo)coercive (circular discs), have a positive semi-definite Hermitian part (region within smaller ellipse), and those for which the solutions of the ODE system $x'=-\mB x$ are stable (region within bigger ellipse), respectively.}
 \label{fig:VennDiagram}
\end{figure}


\begin{remark}\label{rem_Hamiltonian}
As one of the main applications of the analysis of the three discussed concepts is the study of (semi-)dissipative Hamiltonian systems, a natural concept that could be added to the description of the dynamical system is that of a \emph{Hamiltonian or energy function}.
In the abstract setting that we have discussed so far, the natural energy function is the Euclidean norm of the solution.
Further energy functions will be discussed below.
\end{remark}
\begin{remark}[logarithmically optimal norms]
For a Hermitian matrix $\mAc\in\Cnn$, its logarithmic norm $\mu(\mAc)$ and its spectral abscissa $\alpha(\mAc)$ are equal, $\mu(\mAc)=\alpha(\mAc)$.
In general, however, only the inequality $\alpha(\mAc)\leq \mu(\mAc)$ holds, see e.g. \cite[Lemma 1c]{St75}.
A norm is \emph{logarithmically optimal}
with respect to a matrix $\mAc$ if its spectral abscissa $\alpha(\mAc)$ and logarithmic norm $\mu(\mAc)$ are equal, i.e. $\alpha(\mAc) =\mu(\mAc)$.
Thus the Euclidean norm is logarithmically optimal for all Hermitian matrices.
\end{remark}

To analyze the relationship between the different concepts further, in the next section we first discuss the question by which transformations of~\eqref{ODE:B} we can switch between the different concepts and which transformations leave the different properties invariant.

\subsection{Linear transformations that preserve stability, semi-dissipativity, and hy\-po\-co\-er\-ci\-vi\-ty}\label{sec:trafo}
In this section we discuss the classes of linear transformations that preserve the concepts of stability, semi-dissipativity, and hypocoercivity, and also those that map between the different concepts, see also e.g. \cite{JoSm05, JoSm06} for some references.
The natural classes of linear transformations that preserve the different properties and the HC-index (in case of accretive matrices) are
\emph{conjugate transposition} $\mB\to\mB^\topH$, due to Definition~\ref{def:HCI} and Lemma~\ref{lem:HC:equivalence};
\emph{unitary congruence transformations} $\mB\to\mU\mB\mU^\topH$ for a unitary matrix $\mU$, due to Definition~\ref{def:HCI} and Lemma~\ref{lem:HC:equivalence};
\emph{scaling} $\mB\to t\mB$ for any $t\in\R^+$, due to Definition~\ref{def:HCI} and Lemma~\ref{lem:HC:equivalence};
and, as we will show in Lemma~\ref{lemma:Inversion} below, the \emph{inversion} of accretive hypocoercive matrices.

It is a classical result, see e.g.~\cite{Adr95}, how to construct a similarity transformation of a ``stable'' matrix~$\mB$ such that the transformed matrix is accretive:
The origin $x\equiv0$ is a stable state of system~\eqref{ODE:B} if and only if there exists a positive definite matrix $\mP=\mP^\topH\in\Cnn$ that satisfies the \emph{Lyapunov matrix inequality}
\begin{equation} \label{stable0:P}
 \mB^\topH \mP+\mP\mB\geq 0 \,.
\end{equation}
A congruence transformation with the Hermitian matrix~$\mP^{-1/2}$, i.e.\ the inverse of the positive definite square root of $\mP$, yields
\begin{equation} \label{hatBH}
 0
 \leq \mP^{-1/2} (\mB^\topH \mP +\mP\mB) \mP^{-1/2}
 = \mP^{-1/2} \mB^\topH \mP^{1/2} +\mP^{1/2} \mB \mP^{-1/2}
 = 2 \big( \mP^{1/2} \mB \mP^{-1/2}\big)_H \,.
\end{equation}
Hence, the matrix
\begin{equation} \label{hatB}
 \hB :=\mP^{1/2} \mB \mP^{-1/2}
\end{equation}
is accretive.
Moreover, the change of basis $\tilde  \bx(t) := \mP^{1/2} \bx(t)$ transforms~\eqref{ODE:B} into a semi-dissipative Hamiltonian ODE system of the form
\begin{equation}\label{ODE:hatB}
 \tilde \bx'(t)
 =
 -\big(\mP^{1/2} \mB \mP^{-1/2}\big) \tilde \bx(t)
 =
 -\hB\ \tilde \bx(t)\,.
\end{equation}
%
%

Although \emph{similarity transformations} $\mB \to\mS\mB\mS^{-1}$ for invertible matrices~$\mS\in\Cnn$ preserve the spectrum (and hence (negative) hypocoercivity), they \emph{may change} the HC-index of accretive matrices: 
%
\begin{example} 
The matrix
\begin{equation}\label{ODE:B:envelope}
 \mB :=\begin{bmatrix} 1 & -1 \\ 1 & 0 \end{bmatrix}
\end{equation}
is accretive and hypocoercive with $\mHC=1$ (having eigenvalues $\lambda_\pm =(1\pm i\sqrt{3})/2$).
The positive definite Hermitian matrix~$\mP=\begin{bmatrix} 2 & -1 \\ -1 & 2 \end{bmatrix}$ satisfies the continuous-time Lyapunov equation $\mB^\topH\mP +\mP\mB =2\Re(\lambda) \mP =\mP$. 
The similarity transformation~\eqref{hatB} yields a coercive matrix
\[
 \hB
 =\mP^{1/2}\mB\mP^{-1/2}
 =\tfrac12 \begin{bmatrix} 1 & -\sqrt{3} \\ \sqrt{3} & 1 \end{bmatrix}\,,
\]
hence $\mHC(\hB)=0$.
\end{example}

\bigskip
In a similar way, \emph{non-unitary} congruence transformations $\mB \to\mQ\mB\mQ^\topH$ for some nonsingular matrix~$\mQ\in\Cnn$ may change the HC-index as the following example demonstrates.
\begin{example}
Consider the accretive matrix
\[ \mB =\begin{bmatrix} i & 0 \\ 0 & 1 \end{bmatrix} \,. \]
The matrix~$\mB$ has an eigenvalue $i$, hence it is not hypocoercive. 
A congruence transformation with the (non-unitary) matrix
\begin{equation*}
\mQ
 =\begin{bmatrix} 1 & 0 \\ 1 & 1 \end{bmatrix}
 \quad\text{ yields }\quad
\mQ\mB\mQ^\topH
 =\begin{bmatrix} i & i \\ i & 1+i \end{bmatrix}
 =\begin{bmatrix} 0 & 0 \\ 0 & 1 \end{bmatrix}
  +i \begin{bmatrix} 1 & 1 \\ 1 & 1 \end{bmatrix} \,,
\end{equation*}
which is again accretive (due to Sylvester's inertia theorem, see e.g.\ \cite{Gan59a}).
However, the matrix $\mQ\mB\mQ^\topH$ has eigenvalues $\tfrac12 +i\ (1\pm\tfrac{\sqrt {3}}{2})$, and is hypocoercive with HC-index~$\mHC=1$.
\end{example}

As we have already discussed, changing the HC-index also changes the short-time behavior of the solutions of the dynamical system~\eqref{ODE:B}.

\begin{example}
Consider the matrix $\mB$ in Example~\ref{ex:non-coerciveB}.
In agreement with Proposition~\ref{prop:ODE-short}, (the norm of) solutions of the ODE~\eqref{ODE:B} may have horizontal tangents (at any point~$t_0\ge0$) with local behavior $\|x(t)\| =\|x(t_0)\| -c (t-t_0)^3 +\bigO((t-t_0)^4)$ for some $c>0$.
Proceeding as in \cite[Lemma 4.3]{ArEr14}, the similarity transformation~\eqref{hatB} with
\[
 \mP
 =\begin{bmatrix} 3 & 2 \\ 2 & 3 \end{bmatrix}
 \quad\text{yields a coercive matrix}\quad
 \hB
 =\mP^{1/2}\mB\mP^{-1/2}
 =\begin{bmatrix} 1 & \sqrt5 \\ -\sqrt5 & 1 \end{bmatrix}\,.
\]
Accordingly, (the norm of) solutions of the associated ODE~\eqref{ODE:hatB} cannot have horizontal tangents (see Figure~\ref{fig:ODE-decay}).
\begin{figure}
\includegraphics[width=\textwidth]{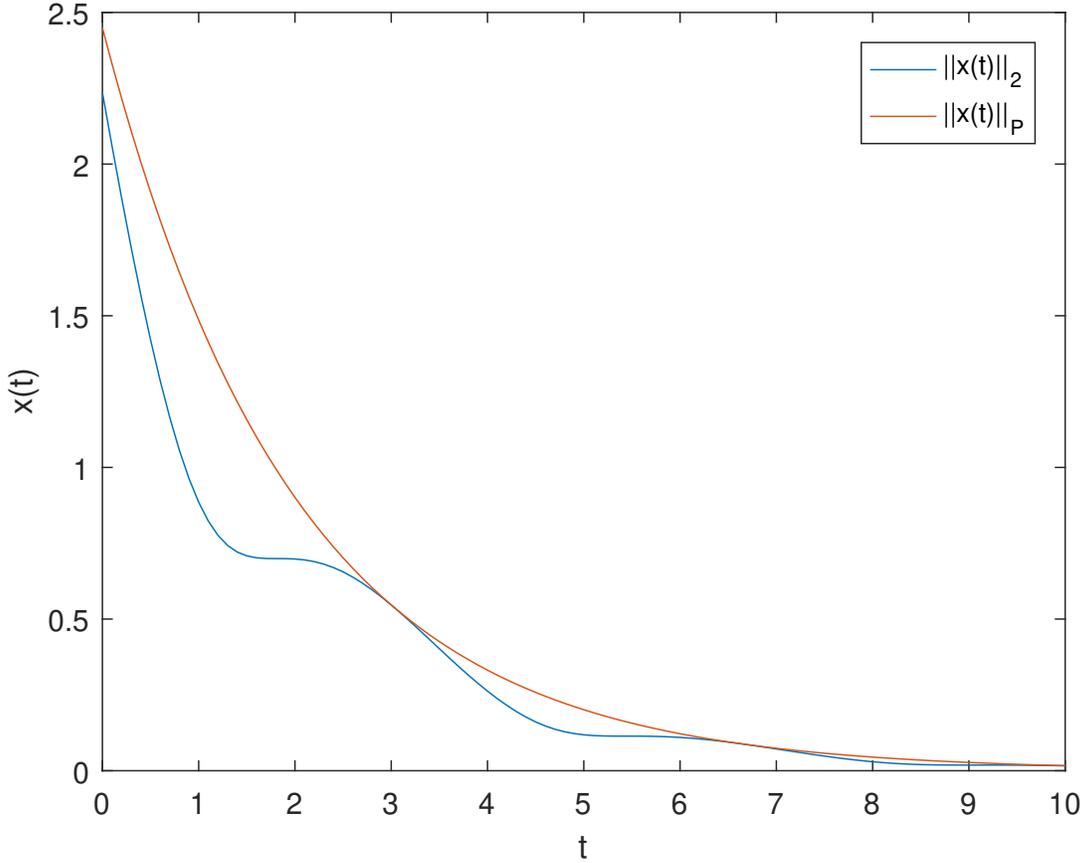}
\caption{For a solution $x(t)$ of the ODE~\eqref{ODE:B} with $\mB=\begin{bmatrix} 1 & -1 \\ 1 & 0 \end{bmatrix}$, the Euclidean norm $\|x(t)\|_2$ (blue line) and the weighted Euclidean norm $\|x(t)\|_{\mP}$ with $\mP=\begin{bmatrix} 2 & -1 \\ -1 & 2 \end{bmatrix}$ (orange line) are plotted.
The norm of the solution $\|x(t)\|_2$ has horizontal tangents (at some point~$t_0$), whereas the weighted norm $\|x(t)\|_{\mP}$ does not have horizontal tangents (due to our choice of $\mP$).}
 \label{fig:ODE-decay}
\end{figure}
\end{example}

\begin{remark}
We note that solutions $\mP$ to the Lyapunov inequality \eqref{stable0:P} are typically not unique, and one can use this freedom to determine solutions that optimize certain robustness measures like the distance to instability, see e.g.\ \cite{BanMNV20,MR3859144,MehV20a}.
\end{remark}

\bigskip
It is an important observation that semi-dissipativity, hypocoercivity and the  HC-index stay invariant when the inverse of a matrix is taken:
\begin{lemma} \label{lemma:Inversion}
Let $\mB\in\Cnn$. 
\begin{itemize}
 \item [1.]
If $\mB$ is hypocoercive then $\mB$ is invertible and $\mBI$ is hypocoercive.
 \item [2.]
If $\mB$ is accretive and invertible then it follows that
\begin{itemize}
 \item [a.]
If $v\in\ker(\mBH)\subset\Cn$ then $\mB v\in\ker(\mBIH)$.
 \item [b.]
$\mBI$ is accretive.
 \item [c.]
$\dim\ker(\mBH) =\dim\ker(\mBIH)$.
\end{itemize}
 \item [3.]
If $\mB$ is accretive and hypocoercive then $\mB$ and $\mBI$ have the same HC-index.
\end{itemize}
\end{lemma}
%
\begin{proof}

1. 
A matrix~$\mB$ is hypocoercive if all eigenvalues 
have positive real-part.
Hence, the matrix~$\mB$ is invertible, and since the eigenvalues of the inverse of~$\mBI$ are the inverses of the eigenvalues of $\mB$, they 
have positive real-part and $\mBI$ is hypocoercive.

2a. 
Writing~$\mB$ as $\mB=\mBH+\mBS$, 
it follows that if $v\in\ker(\mBH)$ then $\mB v =\mBS v =-\mB^\topH v$.
Thus,
\begin{equation}\label{AIH*Av}
 \mBIH (\mB v)
 =
 \tfrac12
 \big( \mBI (\mB v) +(\mBI)^\topH (\mB v)\big)
 =
 \tfrac12
 \big( v -(\mB^\topH)^{-1} (\mB^\topH v)\big)
 =
 0 \ .
\end{equation}

2b. 
To prove that $\mBI$ is again accretive, we show the following identity:
For all vectors $w\in\Cn$, define $v:=\mBI w$, such that
\begin{equation}
\begin{split}
\ip{w}{\mBIH w}
&=
\tfrac12 \ip{w}{(\mBI +(\mBI)^\topH) w}
=
\tfrac12 \ip{\mB v}{(\mBI +\mB^{-\topH}) \mB v}
\\
&=
\tfrac12 \ip{v}{\mB^\topH (\mBI +\mB^{-\topH})\mB v}
=
\tfrac12 \ip{v}{(\mB^\topH +\mB) v}
=
\ip{v}{\mBH v}
\geq
0 \ ,
\end{split}
\end{equation}
since $\mB$ is accretive.
Hence, $\mBI$ is accretive as well.

2c. Due to part 2a. and a similar statement with the roles of $\mB$ and $\mBI$ exchanged, $\mB$ is a bijection from $\ker(\mBH)$ to $\ker(\mBIH)$.


3. By assumption, the matrix~$\mB$ has a finite HC-index~$\mHC$ which is the smallest integer such that \eqref{condition:KalmanRank:BS_BH} holds 
or equivalently, due to \eqref{Tm:BS_BH2}, that
\[
 \bigcap_{j=0}^{\mHC} \ker \big(\mBH \mB^j\big)
 =
 \{0\}
\]
holds, see also~\cite[Remark 4]{AAC18}.
Hence, there exists a vector $v_0\neq 0$ such that
\begin{equation} \label{v0}
 \mB^j v_0\in \ker (\mBH) \ ,
 \qquad
 j\in\{0,\ldots,\mHC-1\}
 \qquad \text{and }
 \mB^{\mHC} v_0 \notin \ker(\mBH) \ . 
\end{equation}
Due to 2b., 
it follows that
\begin{equation} \label{condAv0}
 \mB^{j+1} v_0\in \ker(\mBIH) \ ,
 \qquad
 j\in\{0,\ldots,\mHC-1\}
 \qquad \text{and }
 \mB^{\mHC+1} v_0 \notin \ker(\mBIH) \ .
\end{equation}
The matrix $\mBI$ is hypocoercive and accretive with finite HC-index~$\mHCmBI:=\mHC(\mBI)$ and hence, there exists a vector $w_0\neq 0$ such that
\begin{equation} \label{condw0}
 (\mBI)^j w_0\in \ker(\mBIH) \ ,
 \quad
 j\in\{0,\ldots,\mHCmBI-1\}
 \quad \text{and }
 (\mBI)^{\mHCmBI} w_0 \notin \ker(\mBIH) \ .
\end{equation}
To show that $\mHC =\mHC(\mB) =\mHC(\mBI) =\mHCmBI$, suppose that $v_0$ is a vector in $\Cn$ satisfying~\eqref{condAv0} with $\mHC=\mHC(\mB)$.
Then $w_0 :=\mB^{\mHC} v_0$ satisfies $w_0\ne0$, hence, \eqref{condw0} implies that $\mHC(\mB^{-1})\geq \mHC(\mB)$.
Exchanging the roles of $\mB$ and $\mBI$ shows that $\mHC(\mB^{-1})\leq \mHC(\mB)$.
Altogether, $\mHC(\mB)=\mHC(\mBI)$ holds.
\end{proof}

In this section we have discussed linear transformations and their effects on the concepts of hypocoercivity, stability and semi-dissipativity.
In the next section we discuss how the (concept of the) HC-index for accretive matrices can be transferred to general matrices.

\subsection{Shifted hypocoercivity index for general matrices}
\label{sec:SHCIndex}
A possibility to turn a general system \eqref{ODE:B} into a semi-dissipative Hamiltonian system is to shift the spectrum.
%
Consider the transformation
\begin{equation}\label{shift}
\bv(t) := \exp(\lambdaMinBH t) \bx(t) \,,
\end{equation}
where $\lambdaMinBH$ is the minimal (real) eigenvalue of the Hermitian matrix~$\mBH$.
Then, $\bv(t)$ satisfies the ODE
\[
 \bv'(t)
 = -\underbrace{(\mB -\lambdaMinBH \mI)}_{=:\tB} \bv(t) \ ,
\]
where the Hermitian part~$\tB_H$ of~$\tB =\mB -\lambdaMinBH \mI$ is indeed positive semi-definite.
Of course, the hypocoercivity index of matrix $\tB$ is typically modified by the shift parameter~$\lambda$.
\begin{remark}
The transformation \eqref{shift} can be motivated as follows:
The propagator for ODE~\eqref{ODE:B} with $\mAc=-\mB$ satisfies estimate~\eqref{ODE:short-t} based on the logarithmic norm~$\mu(\mAc)$.
Therefore, for $t\geq 0$,
\[
 1
\geq \|e^{\mAc t}\| e^{-\mu(\mAc)\ t}
= \| e^{(\mAc -\mu(\mAc)\mI)\ t}\|
= \| e^{-(\mB -\lambdaMinBH\mI)\ t}\|,
\]
since the logarithmic norm $\mu(\mAc)$ can also be characterized as
\[
 \mu(\mAc)
:= \sup_{\|x\|=1} \Re (\ip{x}{\mAc x})
 = \sup_{\|x\|=1} \ip{x}{\tfrac12 (\mAc^\topH +\mAc) x}
 = \lambda_{\max}^{\mAH}
 = -\lambda_{\min}^{\mBH} \ ,
\]
where $\lambdaMaxAH$ is the maximal (real) eigenvalue of the Hermitian matrix~$\mAH$.
\end{remark}

%
In view of this shifting property, for general linear time-invariant ODE systems~\eqref{ODE:B} with matrix $\mB\in\Cnn$, we will define a \emph{shifted hypocoercivity index} 
which characterizes ``the algebraic factor`` in the decay of its propagator norm for short time, see Corollary~\ref{cor:mB:SHC-decay} below.
%
As a first step, we decompose the matrix $\mB\in\Cnn$.
\begin{lemma} \label{lem:mB:decomp}
Let $\mB\in\Cnn$ with Hermitian part $\mBH$, and let $\lambdaMinBH$ be the minimal (real) eigenvalue of the Hermitian matrix~$\mBH$ (which could be negative or non-negative).
Then,
the matrix
\begin{equation} \label{mB:decomp}
 \tB := \mB -\lambdaMinBH \mI
\end{equation}
is accretive and, if~$\tB$ is hypocoercive, has an HC-index $\mHC(\tB)$ greater than~$0$.
\newline\indent
In particular, 
$\tB$ is hypocoercive if and only if no eigenvector of~$\mBH$ associated with $\lambdaMinBH$ is an eigenvector of the skew-Hermitian part $\mBS$ of $\mB$.
\end{lemma}
\begin{proof}
If we decompose $\mB=\mBH +\mBS$ into its Hermitian part $\mBH$ and its skew-Hermitian part $\mBS$, then $\mBH$  has only real eigenvalues.
Consider the matrix $\tB := \mB -\lambda \mI$ for $\lambda\in\R$.
Then $\lambda=\lambdaMinBH$ is the only value for which the Hermitian part of $\tB$ is positive semi-definite and singular (hence, if~$\tB$ is hypocoercive then $\mHC(\tB)>0$).
\newline\indent
The hypocoercivity condition for $\tB$ follows from Lemma \ref{lem:HC:equivalence}, (B3):
Matrix~$\tB$ fails to be hypocoercive 
if and only if an eigenvector~$v$ of $\mBS$ (which is not changed by the shift) is in the kernel of $(\tB+\tB^\topH)/2 =\mBH -\lambdaMinBH\mI$, or equivalently $v$ is an eigenvector of $\mBH$ to the eigenvalue~$\lambdaMinBH$.
\end{proof}

\begin{definition}\label{def-SHC-index}
Let~$\mB\in\Cnn$ with Hermitian part~$\mBH$, and let $\lambdaMinBH$ be the minimal (real) eigenvalue of the Hermitian matrix~$\mBH$.
If the accretive matrix $\tB := \mB -\lambdaMinBH \mI$ is hypocoercive, then its HC-index~$\mHC\in\N$ is called the \emph{shifted hypocoercivity index (SHC-index)~$\mSHC$} of~$\mB$.
\end{definition}

By definition, an accretive matrix has a (finite) HC-index~$\mHC$ if and only if it is positively stable, see also~\cite{AAC22,AAM21}.
However, a general (constant) matrix can have a finite SHC-index~$\mSHC$ without being positively stable, see the following example
and Figure~\ref{fig:VennDiagram}.
%
\begin{example} \label{ex:mSHC:unstable}
Consider the matrix
\[
 \mB :=\begin{bmatrix} 9 & -3 \\ 3 & -1 \end{bmatrix}
\]
which has the eigenvalues $\lambda_1 =0$ and $\lambda_2 =8$ and hence is not positively stable.
Its Hermitian part $\mBH=\diag(9,-1)$ has the minimal eigenvalue $\lambdaMinBH=-1$.
Then, in 
\eqref{mB:decomp} we have
\[
 \tB
 =\mB -\lambdaMinBH \mI
 =\begin{bmatrix} 10 & -3 \\ 3 & 0 \end{bmatrix}
\]
which has eigenvalues $1$ and $9$.
Therefore, $\mSHC(\mB) =\mHC(\tB) =1$.
\end{example}

%
We have
the following characterization for
accretive matrices to have a (finite) SHC-index.
\begin{corollary} \label{cor:SHC:equivalence}
Let $\mJ,\mR\in\Cnn$ satisfy $\mR=\mR^\topH$ and $\mJ=-\mJ^\topH$
and let $\lambda_{\min}$ be the minimal eigenvalue of~$\mR$.
Define $\tR :=\mR-\lambda_{\min} \mI$. 
Then the following conditions are equivalent:
\begin{enumerate}[(B1)]
\item \label{SHC:KRC}
There exists $m\in\N_0$ such that
\begin{equation}\label{condition:KalmanRank:J_R_lambda}
 \rank([\mR,\mJ{\mR},\ldots,\mJ^m {\mR}]-\lambda_{\min}[\mI,\mJ,\ldots,\mJ^m])
 =n \,.
\end{equation}
\item \label{SHC:Tm}
There exists $m\in\N_0$ such that
\begin{equation}\label{Tms:J_R_lambda}
\sum_{j=0}^m \mJ^j \mR (\mJ^\topH)^j > \lambda_{\min} \sum_{j=0}^m \mJ^j (\mJ^\topH)^j\,.
\end{equation}
%
\item \label{SHC:PBH:EVec}
No eigenvector of~$\mJ$ is an eigenvector to $\lambda_{\min}$ of~$\mR$.
\item \label{SHC:PBH:EVal}
$\rank [\lambda \mI-\mJ,\mR-\lambda_{\min} \mI] =n$ for every $\lambda\in\C$, in particular for every eigenvalue $\lambda$ of $\mJ$.
%
\end{enumerate}
Moreover, the smallest possible~$m\in\N_0$ in~(B\ref{SHC:KRC}) 
and~(B\ref{SHC:Tm}) coincide. 
\end{corollary}
\begin{proof}
The Hermitian matrix $\tR =\mR-\lambda_{\min} \mI$ is positive semi-definite.
Hence, the statement (which is stated for the original matrix~$\mR$ using $\tR =\mR-\lambda_{\min} \mI$) follows from Lemma \ref{lem:HC:equivalence}.
\end{proof}

In the following result we show that we can use the SHC-index to characterize the short-time behavior of the propagator norm for general linear time-invariant systems of ODEs.
For this we denote the solution semigroup pertaining to~\eqref{ODE:B} by $S(t):=e^{-\mB t}\in \Cnn$, $t\geq 0$.
\begin{corollary} \label{cor:mB:SHC-decay}
Consider an ODE~\eqref{ODE:B} with system matrix~$\mB\in\Cnn$.
If $\mB$ has a finite SHC-index~$\mSHC(\mB)$, then
\begin{equation}\label{mB:short-t-decay}
 \|e^{-\mB t}\|_2
 =
 e^{-\lambdaMinBH t}\
 \big( 1 -ct^a +\bigO(t^{a+1})\big)
 \quad\text{for } t\to0+\,,
\end{equation}
where $\lambdaMinBH$ is the smallest eigenvalue of the Hermitian matrix $\mBH$, $a=2\mSHC(\mB)+1\ (\geq 3)$, and $c>0$.
\end{corollary}
\begin{proof}
Write $\mB$ as in~\eqref{mB:decomp} and compute the HC-index~$\mHC(\tB)\ (\geq 1)$ of the accretive matrix~$\tB =\mB -\lambdaMinBH\mI$. 
Using the decomposition~\eqref{mB:decomp} yields
\begin{equation}
 e^{-\mB t}
 =
 e^{-(\lambdaMinBH \mI +\tB) t}
 =
 e^{-\lambdaMinBH t}\ e^{-\tB t}
 \ , \text{ such that }
 \|e^{-\mB t}\|_2
 =
 e^{-\lambdaMinBH t}\ \|e^{-\tB t}\|_2 \ .
\end{equation}
If an accretive matrix $\tB$ is hypocoercive, i.e.\ having a finite HC-index~$\mHC(\tB)$ (or equivalently~$\mB$ has a finite SHC-index~$\mSHC(\mB)$) then~\eqref{mB:short-t-decay} follows from Proposition~\ref{prop:ODE-short}.
\end{proof}

In this section we have gathered and extended results about stable, hypocoercive, and semi-dissipative matrices.
These results have analoga for discrete-time systems that are studied in the next section.


\section{Stability, semi-contractivity and hypocontractivity for discrete-time systems}
\label{sec:Stability+ODEs:discrete-time}

In this section we study the analogous concepts for linear discrete-time systems 
\begin{equation}\label{DS:A}
 x_{k+1} = \mAd x_k  \,, \qquad k\in\N_0\,,
\end{equation}
for a given matrix $\mAd\in\Cnn$.

\begin{remark}\label{rem:discrete_concepts}
While the stability analysis in discrete-time systems is well studied in linear algebra and operator theory \cite{LaS86} using spectral properties and discrete-time Lyapunov equations, we proceed by studying hypo\-con\-trac\-tiv\-i\-ty---the analogon to the concept of hypocoercivity in continuous time---and relating to these classical concepts.
\end{remark}

\begin{definition}\label{def:disstab}
The trivial solution $x\equiv 0$ of the discrete-time system~\eqref{DS:A} is called \emph{stable} if all solutions of~\eqref{DS:A} are bounded for $k\in\N_0$, and it is called \emph{asymptotically stable} if it is stable and all solutions of~\eqref{DS:A} converge to $0$ for $k\to \infty$.
\end{definition}

For linear systems~\eqref{DS:A} a solution is (asymptotically) stable if and only if the trivial solution $x\equiv 0$ is (asymptotically) stable.
Therefore, if the trivial solution $x\equiv 0$ of~\eqref{DS:A} is (asymptotically) stable then the linear system~\eqref{DS:A} is called (asymptotically) stable.
%

It is well-known that~\eqref{DS:A} is stable if all eigenvalues of~$\mAd$ have modulus less or equal than one and the eigenvalues of modulus one are semi-simple (see~\cite[Theorem 3.3.11]{HiPr10}); and it is \emph{asymptotically stable} if all eigenvalues of~$\mAd$ have modulus strictly less than one.
\begin{definition} 
Let $\mAd\in\Cnn$ have eigenvalues $\lambda_j$, $j=1,\ldots,n$.
The \emph{spectral radius of~$\mAd$} is defined as $\rho(\mAd) :=
\max\{|\lambda_1|\,,\ldots\,,|\lambda_n|\}$, i.e.\ as the largest absolute value of its eigenvalues.
\end{definition}

Hence, a discrete-time system~\eqref{DS:A} is \emph{asymptotically stable} if the spectral radius of $\mAd$ is strictly less than one, $\rho(\mAd)<1$.

An alternative characterization of (asymptotic) stability can be given via the \emph{discrete-time Lyapunov (\emph{or} Stein) equation}:
System~\eqref{DS:A} is asymptotically stable if and only if, for all positive definite Hermitian matrices $\mQ$
\begin{equation}\label{discrete:Lyapunov}
 \mAd^\topH \mP \mAd -\mP =-\mQ
\end{equation}
has a solution $\mP=\mP^\topH>0$, see~\cite[Theorem 3.3.49]{HiPr10} which is formally given by
\begin{equation}\label{discrete:Lyapunov:solution}
 \mP = \sum_{j=0}^\infty (\mAd^\topH)^j \mQ \mAd^j \ ,
\end{equation}
see~\cite[{(89b) in \S3.3.5}]{HiPr10}.
In the discrete-time case the concept of hypocoercivity is replaced by that of hypocontractivity, which we introduce in the next subsection.

\subsection{Hypocontractive matrices and the hypocontractivity index}

For~$\mAd \in\Cnn$ the spectral norm satisfies
\begin{equation}
 \|\mAd\|_2
 =
 \sqrt{\|\mAd^\topH \mAd\|_2}
 =
 \sqrt{\lambda_{\max} \big(\mAd^\topH \mAd\big)}
 =
 \sigma_{\max} (\mAd) \ ,
\end{equation}
where $\lambda_{\max} \big(\mAd^\topH \mAd\big)$ denotes the largest eigenvalue of the positive semi-definite Hermitian matrix $\mAd^\topH \mAd$ and $\sigma_{\max} (\mAd)$ is the largest singular value of $\mAd$.
Then, the estimate $\|\mAd^n\|_2 \leq \|\mAd\|_2^n$ for $n\in\N$ yields that $\sigma_{\max} (\mAd) \leq 1$ is a sufficient condition for the stability of~\eqref{DS:A}.
However, $\sigma_{\max}(\mAd)\leq 1$ is not a necessary condition for~\eqref{DS:A} to be stable.
\begin{example}
The eigenvalues of
\begin{equation} \label{mAd:lambda}
 \mAd(\alpha)
 =
 \alpha
 \begin{bmatrix}
  1 & -2 \\
  0 & -1
 \end{bmatrix}\,,
 \quad \alpha\in\R\,,
\end{equation}
are $\pm\alpha$.
Hence, the discrete-time system~\eqref{DS:A} with matrix~$\mAd$ in~\eqref{mAd:lambda} is stable if and only if $\alpha\in[-1,1]$.
But the matrix
\begin{equation}
 \mAd^\topH \mAd
 =
 \alpha^2
 \begin{bmatrix}
  1 &-2 \\
 -2 & 5
 \end{bmatrix}
\end{equation}
has positive eigenvalues $\mu_{\pm} =\alpha^2 (3\pm\sqrt{8})$ and singular values $\sigma_\pm =\sqrt{\mu_\pm}$ with $\sigma_{\max}(\mAd) =\sigma_+$.
Thus, $\sigma_{\max}(\mAd)\leq 1$ holds if $|\alpha|\leq (3+\sqrt{8})^{-1/2} \leq 1/2$ which is strictly less than one.
Hence in this example, the condition $\sigma_{\max}(\mAd)\leq 1$ is sufficient but not necessary to ensure the stability of~\eqref{DS:A}.
\end{example}

In the following we will need a result relating singular values and eigenvalues.
\begin{proposition} \label{prop:EV+SV}
Let $\mAd\in\Cnn$ have singular values $\sigma_1 \geq \ldots \geq \sigma_n \geq 0$ (such that $\sigma_{\max}(\mAd)=\sigma_1$) and eigenvalues $\lambda_j$, $j=1,\ldots,n$ being ordered as $|\lambda_1| \geq \ldots \geq |\lambda_n|$.
Then, $|\lambda_1|\leq \sigma_1$.
Moreover, if $\mAd$ is nonsingular, then $|\lambda_n|\geq \sigma_n >0$.
\end{proposition}
\begin{proof}
The statements follow from the bounds in~\cite[Theorem 5.6.9]{HoJo13}. 
\end{proof}

We then have the following discrete-time analogon of semi-dissipativity.
\begin{definition}[{\cite[Definition 4.1.2]{Be18}}]
Let $\mAd\in\Cnn$ and let $\sigma_{\max}(\mAd)$ be the largest singular value (the \emph{spectral norm}) of $\mAd$.
We call $\mAd$ \emph{contractive} if $\sigma_{\max}(\mAd) <1$;
and we call $\mAd$ \emph{semi-contractive} if $\sigma_{\max}(\mA) \leq 1$.
\end{definition}

Note that sometimes $\mAd$ is called \emph{contractive} if $\sigma_{\max}(\mAd)\leq 1$;
and $\mAd$ is called \emph{strictly contractive} if $\sigma_{\max}(\mAd)<1$, see e.g.~\cite[p. 493]{HoJo13}.
Other related notions are (semi-)convergent matrices, and power-bounded matrices, see~\cite[p. 180]{HoJo13}.

In the following, we consider the class of semi-contractive matrices~$\mAd$ and present a characterization when~\eqref{DS:A} is (asymptotically) stable.
For this we need a concept that corresponds to hypocoercivity in the continuous-time case.
%
\begin{definition} \label{def:HypoContractive}
A matrix $\mAd\in\Cnn$ is called \emph{hypocontractive} if all eigenvalues of $\mAd$ have modulus strictly less than one.
\end{definition}

Consequently, a discrete-time system~\eqref{DS:A} is asymptotically stable if and only if the system matrix~$\mAd$ is hypocontractive.
We can also characterize those semi-contractive matrices~$\mA$ which are actually hypocontractive:
%
\begin{proposition}\label{prop:Semi+HypoContractive}
Let $\mAd\in\Cnn$ be semi-contractive.
Then, $\mAd$ has an eigenvalue of modulus one if and only if some eigenvector~$v$ of $\mAd$ satisfies~$v\in\kernel(\mI-\mAd^\topH\mAd)$.
\end{proposition}
\begin{proof}
Since $\mAd$ is semi-contractive, the Hermitian matrix~$\mI -\mAd^\topH\mAd$ is positive semi-definite.
Moreover, if~$\mAd$ has an eigenvalue~$\lambda$ of modulus $|\lambda|=1$ with eigenvector~$v\ne 0$, then 
\[
 0
 \leq \ip{v}{(\mI-\mAd^\topH\mAd)v}
 =\|v\|^2 -\|\mAd v\|^2
 =\|v\|^2 (1-|\lambda|^2)
 =0 \,.
\]
Therefore, $v$ is in the kernel of the positive semi-definite Hermitian matrix~$\mI-\mAd^\topH\mAd$.

Conversely, if some eigenvector~$v$ of~$\mAd$ (associated to an eigenvalue~$\lambda$) satisfies $v\in\kernel(\mI-\mAd^\topH\mAd)$, then
\[
 0
 =\ip{v}{(\mI-\mAd^\topH\mAd)v}
 =\|v\|^2 -\|\mAd v\|^2
 =\|v\|^2 (1-|\lambda|^2) \,,
\]
and hence, the eigenvalue~$\lambda$ has modulus one.
\end{proof}

\begin{remark}\label{rem:defect}
In the operator theory setting the matrix $(\mI-\mAd^\topH\mAd)^{1/2}$ is often called the \emph{defect operator} of the semi-contractive $\mA_d$ and the closure of its image is the \emph{defect space} with its dimension being called the \emph{defect index} $d(\mAd)$.  The defect operator and its index are a measure for the distance of an operator from being unitary. See e.g. \cite{NagFK10}.
\end{remark}


We again have an equivalent characterization in terms of properties from control theory:
\begin{lemma} \label{lem:mDHC:Equivalence}
Let $\mAd\in\Cnn$ be semi-contractive.
Then the following conditions are equivalent:
\begin{itemize}
\item [(D1)] 
There exists $m\in\N_0$ such that
\begin{equation}\label{condition:KalmanRank:D}
 \rank[(\mI -\mAd^\topH\mAd),\mAd^\topH (\mI -\mAd^\topH\mAd),\ldots,(\mAd^\topH)^m (\mI -\mAd^\topH\mAd)]=n \,.
\end{equation}
\item [(D2)] 
There exists $m\in\N_0$ such that
\begin{equation}\label{Dm:A}
 \mD_m
 := \sum_{j=0}^m (\mAd^\topH)^j (\mI -\mAd^\topH\mAd) \mAd^j
 >0 \,.
\end{equation}
\item [(D3)] 
No eigenvector of~$\mAd$ lies in the kernel of~$(\mI -\mAd^\topH\mAd)$.
\item [(D4)] 
$\rank [\lambda \mI-\mAd^\topH, \mI -\mAd^\topH\mAd] =n$ for every $\lambda \in \C$, in particular for every eigenvalue~$\lambda$ of~$\mAd^\topH$.
\end{itemize}
Moreover, the smallest possible~$m\in\N_0$ in~(D1) 
and~(D2) 
coincide. 
\end{lemma}
\begin{proof}
Like Lemma \ref{lem:HC:equivalence}, this result follows from Theorem~6.2.1 of~\cite{Da04}  and Lemma \ref{lem:Definiteness} in the Appendix.
\end{proof}

\begin{remark}\label{rem:unobservability}
In control theory, conditions (D1), (D3), and (D4)  in Lemma~\ref{lem:mDHC:Equivalence} are equivalent characterizations of the \emph{controllability} of the pair $(\mAd^\topH, \mI -\mAd^\topH\mAd)$, or the dynamical system
\[
x_{k+1}= \mAd^\topH x_k+ (\mI -\mAd^\topH\mAd)u_k.
\]
There is always also the dual concept of \emph{observability} which in this case would correspond to the controllability of $(\mAd, \mI -\mAd\mAd^\topH)$.
A dual result to Lemma~\ref{lem:mDHC:Equivalence} can then be formulated with this pair. Based on this pair, in \cite{Sta05} a similar result has been derived (in different terminology).
A similar result for the continuous-time case follows from \cite{Sta03}.
\end{remark}

If we compare Lemma~\ref{lem:mDHC:Equivalence} with Lemma~\ref{lem:HC:equivalence}, then we need to substitute $\mBS$ with $\mAd^\topH$, and $\mBH$ with $\mI-\mAd^\topH\mAd$, respectively.
%
Using Lemma~\ref{lem:mDHC:Equivalence} (D2), we then define the hypocontractivity index.  
\begin{definition}
For semi-contractive matrices $\mAd\in\Cnn$, 
we define the \emph{hypocontractivity index} or \emph{discrete HC-index (dHC-index)~$\mDHC$} as the smallest integer $m\in\N_0$ (if it exists) such that~\eqref{Dm:A} holds.
\end{definition}

\begin{remark}\label{rem:alternative notions}
The  hypocontractivity index is sometimes also called the \emph{norm-one index}, see \cite{GauW10}, where it is shown that this index is finite if and only if the spectral radius of $\mAd$ is strictly smaller than one.
\end{remark}

Clearly, a semi-contractive matrix $\mAd$ is contractive if and only if $\mDHC=0$.
Since \eqref{Dm:A} is a telescopic sum, we have that $\mD_m =\mI -(\mAd^\topH)^{m+1} \mAd^{m+1}$ and thus if a semi-contractive matrix $\mAd\in\Cnn$ is hypocontractive with index~$\mDHC\in\N_0$, then $\mAd^{\mDHC+1}$ is contractive.
Conversely, if a semi-contractive matrix $\mAd\in\Cnn$ satisfies that $\mAd^m$ is contractive for some $m\in\N$, then $\mAd$ is hypocontractive with index $\mDHC\le m-1$.\\

The following result may be considered as a discrete counterpart of the short-time decay behavior from Proposition \ref{prop:ODE-short}.
%
%
%
\begin{theorem}\label{DS:A:semi-contractive:power-bounds}
Let $\mAd\in\Cnn$ be semi-contractive and hypocontractive.
Its (finite) hypocontractivity index is $\mDHC\in\N_0$ if and only if
\begin{equation} \label{DS:A:decay}
 \|\mAd^j\|_2 =1 \text{ for all } j=1,\ldots,\mDHC\,,
 \quad\text{and }\
 \|\mAd^{\mDHC+1}\|_2 <1 \,.
\end{equation}
\end{theorem}
\begin{proof}
The spectral norm $\|\mC\|_2$ of a matrix $\mC\in\Cnn$, i.e.\ the operator norm induced by the Euclidean norm on~$\Cn$ is given by
$\|\mC\|_2 =\max_{w\in\Cn:\ \|w\|=1} \|\mC w\|_2$.
If a matrix $\mAd$ is semi-contractive, then the estimates $\|\mAd\|_2\leq 1$ and $\|\mAd^j\|_2 \leq\|\mAd\|_2^j\leq 1$ hold for all $ j\in\N$.
Thus, for vectors $w\in\Cn$ with $\|w\|_2=1$, we have
\[
 \ip{w}{(\mAd^\topH)^j \mAd^j w} =\ip{\mAd^j w}{\mAd^j w} =\|\mAd^j w\|_2^2
\leq
 1 =\ip{w}{w} \,,
\]
such that $0\leq \ip{w}{(\mI -(\mAd^\topH)^j \mAd^j) w}$. 
Therefore, for all $m\in\N_0$,
\[
 \mD_m
 =\sum_{j=0}^m (\mAd^\topH)^j (\mI -\mAd^\topH\mAd) \mAd^j
 =\mI -(\mAd^\topH)^{m+1} \mAd^{m+1}
 \geq 0 \,
\]
and hence, the semi-contractive matrix $\mAd$ has (finite) hypocontractivity index $\mDHC$ if and only if~\eqref{DS:A:decay} holds.
\end{proof}

We summarize the relationship between the different concepts discussed in this section in Figure~\ref{fig:VennDiagramA}.

\begin{figure}
\includegraphics[width=\textwidth]{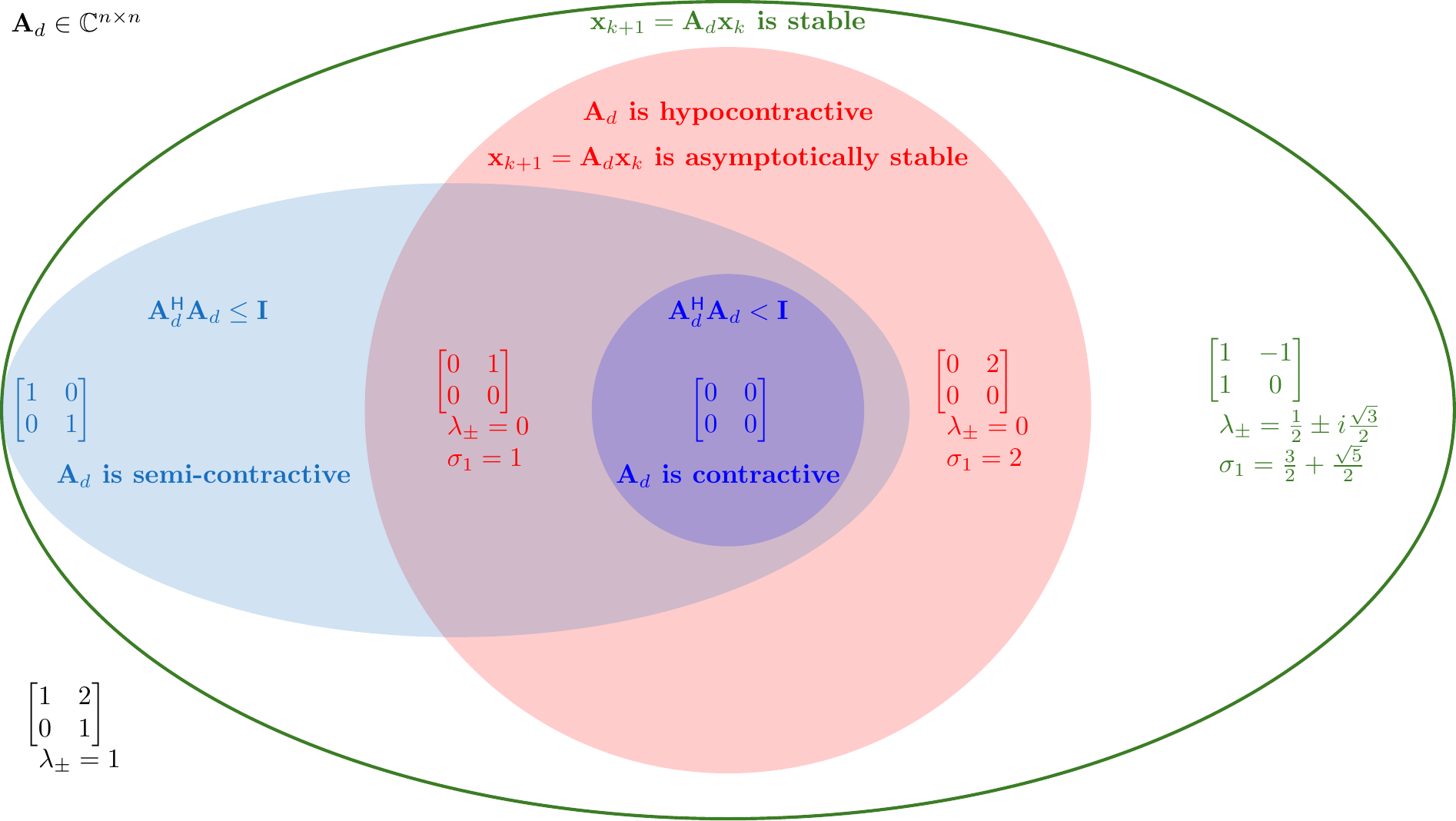}
\caption{Relationship between sets of matrices $\mAd\in\Cnn$ which are (hypo)contractive (circular discs), semi-contractive (region within smaller ellipse) and those for which the discrete-time system $x_{k+1}=\mAd x_k$ is stable (region within bigger ellipse), respectively.}
 \label{fig:VennDiagramA}
\end{figure}
%

\subsection{Polar decomposition}\label{ssec:DS:PD}

In \cite{AAM21} a computationally feasible procedure  has been presented to check the conditions of Lemma~\ref{lem:HC:equivalence} in the continuous-time case via a staircase form under unitary congruence transformations.
A similar procedure can be derived in the discrete-time case.
It is based on polar decomposition, see e.g.~\cite[Theorem 7.3.1]{HoJo13}, which is the discrete-time analogon of the additive splitting of a matrix into its Hermitian and skew-Hermitian part:
%
\begin{proposition}[Polar decomposition]\label{prop:PolDec}
Let $\mAd\in\Cnn$.
\begin{itemize}
\item [(a)]
There exist positive semi-definite Hermitian matrices $\mPd,\mQd\in\Cnn$ and a unitary matrix~$\mUd\in\Cnn$ such that
\begin{equation}\label{PD:A}
 \mAd
 =\mPd \mUd
 =\mUd \mQd \ .
\end{equation}
The factors $\mPd$, $\mQd$ are uniquely determined as $\mPd =(\mAd \mAd^\topH)^{1/2}$ and $\mQd =(\mAd^\topH \mAd)^{1/2}$.
If $\mAd$ is nonsingular, then $\mUd =\mPd^{-1}\mAd =\mAd\mQd^{-1}$ is uniquely determined (as well).
\item [(b)]
If $\mAd$ is real, then the factors $\mPd$, $\mQd$ and $\mUd$ may be taken to be real.
\end{itemize}
\end{proposition}

Consider a stable discrete-time system~\eqref{DS:A} with matrix~$\mAd$.
Hence, all eigenvalues of matrix~$\mAd$ have modulus less or equal than one.
Then, the polar decomposition~\eqref{PD:A} yields that the (largest) singular values of $\mAd$, $\mPd$ and $\mQd$ are the same, since $\mAd \mAd^\topH =\mPd \mPd^\topH$ and $\mAd^\topH \mAd =\mQd^\topH \mQd$. 

An immediate consequence is that a matrix
$\mAd\in\Cnn$ with polar decomposition~\eqref{PD:A}
is semi-contractive if and only if the spectra of $\mPd$ and $\mQd$ (which coincide) are contained in $[0,1]$.
%
%
We can rephrase the statement of Proposition~\ref{prop:Semi+HypoContractive} as follows:
\begin{proposition}\label{prop:Semi+HypoContractive:UQ}
Let $\mAd\in\Cnn$ be semi-contractive with polar decomposition $\mAd =\mUd\mQd$ and $\mQd =(\mAd^\topH \mAd)^{1/2}$.
Then, $\mAd$ has an eigenvalue of modulus one (and hence $\mAd$ is not hypocontractive)
if and only if some eigenvector~$v$ of $\mUd$ satisfies~$v\in\kernel(\mI-\mQd)$.
\end{proposition}
\begin{proof}
For the forward direction we assume that the eigenvalue equation $\mAd v=\lambda v$ holds for some $\lambda$ with $|\lambda|=1$ and $v\in\Cn\setminus\{0\}$.
Then, Proposition~\ref{prop:Semi+HypoContractive} implies that $v\in\ker(\mI-\mAd^\topH\mAd)$, i.e.\ $0=(\mI-\mAd^\topH\mAd)v = (\mI+\mQd)(\mI-\mQd)v$ which holds if and only if $0=(\mI-\mQd)v$, such that  $0=\mUd (\mI-\mQd)v=\mUd v-\lambda v$.
Hence, $v\in\kernel(\mI-\mQd)$ is an eigenvector of $\mUd$.

Conversely, let $w$ be an eigenvector of $\mUd$, i.e.\ $ \mUd w=\lambda w$ with $|\lambda|=1$, that satisfies $(\mI-\mQd)w=0$.
Then $0=\mUd (\mI-\mQd)w=\lambda w -\mAd w$.
\end{proof}

Note that, for semi-contractive matrices~$\mAd$, eigenvalues with modulus one are necessarily semi-simple.
Therefore, a semi-contractive matrix~$\mAd$ (with polar decomposition $\mAd=\mUd\mQd$) is hypocontractive if and only if no eigenvector of~$\mAd$ lies in the kernel of the positive semi-definite Hermitian matrix~$\mI -\mQd$.

Using this relationship,  we formulate an analogous result to Lemma~\ref{lem:mDHC:Equivalence}, in terms of matrices appearing in polar decompositions. It follows again from Theorem 6.2.1 of \cite{Da04} and Lemma \ref{lem:Definiteness}:

\begin{lemma} \label{lem:mDHC:Equivalence:HU}
Let $\mAd\in\Cnn$ be semi-contractive with polar decomposition $\mAd =\mUd\mQd$ (i.e.\ with $\mUd$ unitary, $\mQd$ semi-contractive Hermitian, and $\mQd^2=\mAd^\topH\mAd$).
Then the following conditions are equivalent:
\begin{itemize}
\item [(D1')] 
There exists $m\in\N_0$ such that
\begin{equation}\label{condition:KalmanRank:HU}
 \rank[(\mI-\mQd^2),\mUd^\topH(\mI-\mQd^2),\ldots,(\mUd^\topH)^m (\mI-\mQd^2)]=n \,.
\end{equation}
\item [(D2')]
There exists $m\in\N_0$ such that
\begin{equation}\label{Dm:HU}
 \hD_m :=\sum_{j=0}^m (\mUd^\topH)^j (\mI-\mQd^2) \mUd^j > 0 \,.
\end{equation}
\item [(D3')] 
No eigenvector of~$\mUd$ lies in the kernel of $\mI-\mQd^2$.
\item [(D4')] 
$\rank [\lambda \mI-\mUd^\topH, \mI-\mQd^2] =n$ for every $\lambda \in \C$, in particular for every eigenvalue~$\lambda$ of~$\mUd^\topH$.
\end{itemize}
Moreover, the smallest possible~$m\in\N_0$ in (D1') and (D2')
coincide. 
\end{lemma}

Note that (D3) and (D3') are equivalent, due to Proposition \ref{prop:Semi+HypoContractive:UQ} and since $\ker(\mI-\mQd)=\ker(\mI-\mQd^2)$.
Consequently, all conditions of the Lemmata  \ref{lem:mDHC:Equivalence} and  \ref{lem:mDHC:Equivalence:HU} are equivalent and the smallest possible values of $m$ coincide.

\subsection{Scaled hypocontractivity index}
\label{subsec:SDHC}
The analogon to the shifted hypocoercivity index is obtained by scaling.
\begin{lemma} \label{lem:mA:scale}
Let $\mAd\in\Cnn$ be a nonzero matrix, and let $\sigmaMaxA$ be the maximal singular value of~$\mAd$.
Then, the matrix
\begin{equation} \label{mA:scale}
 \tA_d := (\sigmaMaxA)^{-1} \mAd
\end{equation}
is semi-contractive and, if $\tA_d$ is hypocontractive, has a discrete HC-index $\mDHC(\tA_d)$ greater than~$0$.

Furthermore, $\tA_d$ is hypocontractive if and only if the matrices in the polar decomposition of~$\mAd =\mUd\mQd$ satisfy that no eigenvector of~$\mQd =(\mAd^\topH\mAd)^{1/2}$ associated with the eigenvalue $\sigmaMax{\mAd}$ is an eigenvector of $\mUd$.
\end{lemma}
\begin{proof}
Consider the matrix $\tA_d(\sigma) := \sigma^{-1} \mAd$ for $\sigma>0$.
Then $\sigma=\sigmaMaxA$ is the only value such that the largest singular value of $\tA_d(\sigma)$ is one, since
\[
 \sigmaMax{\tA_d}
 =\sqrt{\lambdaMax{}(\tA_d^\topH \tA_d)}
 =\sqrt{\lambdaMax{}(\mAd^\topH \mAd)}/\sigmaMax{\mAd}
 =1 \,.
\]
Consequently, if the scaled matrix $\tA_d$ is hypocontractive then its discrete HC-index~$\mDHC(\tA_d)$ is greater than $0$. 

To prove the final statement we consider the polar decomposition of~$\mAd$ in the form $\mAd =\mU_d\mQ_d$.
Then, $\tA_d =(\sigmaMaxA)^{-1}\mAd$
has the polar decomposition $\tA_d =\mU_d\tQ_d$ with the same unitary matrix $\mU_d$, and $\tQ_d :=  \linebreak (\sigmaMaxA)^{-1}\mQ_d$.
Due to Proposition~\ref{prop:Semi+HypoContractive:UQ}, $\tA_d$ is hypocontractive if and only if no eigenvector~$v$ of $\mU_d$ is in the kernel of $\mI -\tQ_d$.
The latter 
is equivalent to $v$ being an eigenvector of $\tQ_d$ to the eigenvalue one, or $v$ being an eigenvector of $\mQ_d$ to the eigenvalue~$\sigmaMax{\mAd}$.
\end{proof}
\begin{definition}\label{def-SdHC-index}
Consider a nonzero matrix~$\mAd\in\Cnn$, and let $\sigmaMaxA$ be the maximal (positive) singular value of~$\mAd$.
If the semi-contractive matrix $\tA_d := (\sigmaMaxA)^{-1} \mAd$ is hypocontractive with discrete HC-index~$\mDHC(\tA_d)$ then we define the \emph{scaled hypocontractivity index} or \emph{discrete SHC-index (dSHC-index)~$\mDSHC$} of~$\mAd$ as $\mDSHC(\mAd):=\mDHC(\tA_d)$.
\end{definition}

In analogy to Theorem \ref{DS:A:semi-contractive:power-bounds} we then have the following characterization when $\tA_d$ has a finite scaled hypocontractivity index.
\begin{theorem}\label{DS:A:power-bounds}
Let $\mAd\in\Cnn$ be nonzero, and let $\sigmaMaxA$ be the maximal (positive) singular value of~$\mAd$.
If $\mAd$ has a finite discrete SHC-index~$\mDSHC$, then
\begin{equation} \label{DS:A:decay:general}
 \|\mAd^j\|_2 =(\sigmaMaxA)^j \quad\text{ for all } j=1,\ldots,\mDSHC \ ,
 \ \text{and }
 \|\mAd^{\mDSHC+1}\|_2 <(\sigmaMaxA)^{\mDSHC+1} \ .
\end{equation}
\end{theorem}
\begin{proof}
We scale $\mAd$ as in~\eqref{mA:scale} and compute the discrete HC-index~$\mDHC(\tA_d)\ (\geq 1)$ of the semi-contractive matrix~$\tA_d =(\sigmaMaxA)^{-1} \mAd$ so that $\mDSHC(\mAd) :=\mDHC(\tA_d)$.
Using the scaling~\eqref{mA:scale} yields
\begin{equation}
 \|\mAd^j\|_2
 =\big\|\big(\sigmaMaxA \tA_d\big)^j\big\|_2
 =\big(\sigmaMaxA\big)^j \|\tA_d^j\|_2
 \ \text{for all $j\in\N$.}
\end{equation}

If the semi-contractive matrix $\tA_d$ has a (finite) discrete HC-index~$\mDHC(\tA_d)$ (or equivalently the discrete SHC-index~$\mDSHC(\mAd)$ of~$\mAd$ is finite) then~\eqref{DS:A:decay:general} follows from Theorem~\ref{DS:A:semi-contractive:power-bounds}.
\end{proof}

We summarize the analogy between discrete-time and continuous-time systems in Table~\ref{propinvariance}.

\begin{table}
\begin{tabular}{lll}
properties & continuous-time system & discrete-time system\\[1mm]
\hline\hline\\[-3mm]
evolution & $x' =\mAc x$ for $t\geq 0$ & $x_{k+1} =\mAd x_k $ for $k\in\N_0$ \\
\hline
condition for & $\Re(\lambda)<0$ for all $\lambda\in\Lambda(\mAc)$, & $|\lambda|<1$ for all $\lambda\in\Lambda(\mAd)$, \\
asymptotic stability & i.e.\ negative hypocoercive & i.e.\ hypocontractive\\
\hline
matrix decomposition & $\mAc =\mAS+\mAH$ & polar: $\mAd=\mPd\mUd=\mUd\mQd$ \\
\hline
sufficient stability  & $\mAH\leq 0$, & $\sigma_{\max}(\mAd) \leq 1$ $\Leftrightarrow$ $\Lambda(\mQd)\subset[0,1]$,\\
condition & i.e.\ semi-dissipative& i.e.\ semi-contractive\\
\hline
Kalman rank condition & $\rank[{\mAH}, 
\ldots,((\mAS)^\topH)^m {\mAH}]$ &
$\rank[(\mI-\mQd^2), 
\ldots,(\mUd^\topH)^m (\mI-\mQd^2)]$\\
& $=n$ & $=n$\\
\hline
HC-condition & $\displaystyle\sum_{j=0}^m ((\mAS)^\topH)^j (-\mAH) \mAS^j > 0$ & $\displaystyle\sum_{j=0}^m (\mUd^\topH)^j (\mI-\mQd^2) \mUd^j > 0$\\
\hline
eigenvector condition & no EV of~$\mAS$ in $\ker(\mAH)$ & no EV of~$\mUd$ in $\ker(\mI-\mQd^2)$\\
\hline\\
\end{tabular}
\caption{Relation between concepts for continuous-time and discrete-time systems, see also Figures~\ref{fig:VennDiagram} and~\ref{fig:VennDiagramA}. $\Lambda(\mA)$ denotes here the spectrum of a matrix $\mA\in\Cnn$. \label{propinvariance}}
\end{table}

In this section we have given characterizations for the concepts of stability, semi-contractivity, and hypocontractivity for linear discrete-time systems. In the next section we  relate the  continuous-time and discrete-time concepts.

\section{Transformation between discrete-time and continuous-time systems}\label{sec:dtct}

We have seen the close analogy between the results for the continuous-time and discrete-time case.
In this section we recall that the typical bilinear transformations between continuous-time and discrete-time systems such as the Cayley transformation (in fact of $-\mAc$) relate hypocoercive with hypocontractive systems (see e.g.~\cite{HiPr10}), and semi-dissipative with semi-contractive systems (see e.g.~\cite{NagFK10}).
Moreover, we show that the Cayley transformation (of $-\mAc$) directly relates the hypocoercivity and hypocontractivity indices.

\begin{lemma} \label{lemma:Cayley}
Let $\mAc\in\Cnn$ be a matrix such that~\eqref{ODE:B} is (Lyapunov) stable.
Then, the Cayley transform
\begin{equation}\label{cayley}
  \mAd:=(\mI+\mAc)(\mI-\mAc)^{-1}
\end{equation}
is well-defined and the following properties hold:
\begin{itemize}
 \item [(i)] 
If $\mAc$ is negative hypocoercive then $\mAd$ is hypocontractive.
 \item [(ii)] 
If $\mAc$ is semi-dissipative then $\mAd$ is semi-contractive.
Let $\mAH:= \frac12 (\mAc+\mAc^\topH)$, then the matrix~$(\mI -\mAc)$ is a bijection from~$\ker(\mAH)$ to $\ker(\mI -\mAd^\topH \mAd)$.
Consequently, $\dim\ker(\mAH) =\dim\ker(\mI -\mAd^\topH \mAd)$.
 %
\end{itemize}
\end{lemma}
\begin{proof}
If the continuous-time system~\eqref{ODE:B} with system matrix~$\mAc$ is \emph{(Lyapunov) stable} then all eigenvalues of~$\mAc$ have non-positive real part and the eigenvalues on the imaginary axis are semi-simple.
Hence, the matrices~$(\mI-\mAc)$, $(\mI-\mAc^\topH)$ are invertible; and the  Cayley transform~$\mAd=(\mI+\mAc)(\mI-\mAc)^{-1}$ is well-defined.

(i)  
If~$\mAc\in\Cnn$ is negative hypocoercive, then all eigenvalues of~$\mAd$ have absolute value less than one, hence, $\mAd$ is hypocontractive.

(ii) 
If $\mAc\in\Cnn$ is semi-dissipative, then $\mI-\mAc$ is positive dissipative (hence $\mI-\mAc$ is invertible).
It follows that
\begin{equation} \label{I-AhA}
\begin{split}
 \mI -\mAd^\topH \mAd
&=\mI -(\mI-\mAc)^{-\topH}(\mI+\mAc)^\topH (\mI+\mAc)(\mI-\mAc)^{-1} \\
&=(\mI-\mAc)^{-\topH} \Big((\mI-\mAc)^{\topH}(\mI-\mAc) -(\mI+\mAc)^\topH (\mI+\mAc) \Big) (\mI-\mAc)^{-1} \\
&=-2 (\mI-\mAc)^{-\topH} (\mAc^{\topH} +\mAc) (\mI-\mAc)^{-1} \\
&= -4 (\mI-\mAc)^{-\topH} \mAH (\mI-\mAc)^{-1} \,.
\end{split}
\end{equation}
Hence, the matrices~$-\mAH$ and~$(\mI -\mAd^\topH \mAd)$ are related via a congruence transformation.
Therefore, $\mAd$ is semi-contractive (or equivalently, $(\mI -\mAd^\topH \mAd)$ is positive semi-definite) if $\mAc$ is semi-dissipative.

Due to~\eqref{I-AhA}, if $v\in\ker(\mAH)$ then $(\mI -\mAc) v\in\ker(\mI -\mAd^\topH \mAd)$.
Thus, $(\mI -\mAc)\ker(\mAH) \subseteq\ker(\mI -\mAd^\topH \mAd)$.
Conversely, if $w\in\ker(\mI -\mAd^\topH \mAd)$ then $(\mI -\mAc)^{-1} w\in\ker(\mAH)$.
Thus, $(\mI -\mAc)^{-1} \ker(\mI -\mAd^\topH \mAd)\subseteq \ker(\mAH)$.
Altogether, $(\mI -\mAc)$ is a bijection from $\ker(\mAH)$ to $\ker(\mI -\mAd^\topH \mAd)$, and~$\dim\ker(\mAH) =\dim\ker(\mI -\mAd^\topH \mAd)$.
\end{proof}

\begin{remark}
As a consequence of Lemma \ref{lemma:Cayley}(ii) we have that $\rank(\mAH)=\rank(\mI-\mAd^\topH \mAd) = :d(\mAd)$, the defect index of $\mAd$, see Remark \ref{rem:defect}. As a follow-up consequence (using also Theorem \ref{thm:Cayley} below) we find that the lower bound on the hypocontractivity index of $\mAd$ from \cite{GauW10}, i.e.\ $\mDHC(\mAd)\ge \frac{n-d(\mAd)}{d(\mAd)}$ equals our lower bound on the hypocoercivity index of $\mAc$, i.e.\ $\mHC(\mAc)\ge \frac{n-\rank(\mAH)}{\rank(\mAH)}$.
\end{remark}

The inverse Cayley transform leads to a similar result for the mapping from the discrete-time to the continuous-time problem:

\begin{lemma} \label{lemma:Cayley:Inverse}
Let $\mAd\in\Cnn$ be 
such that $x_{k+1} =\mAd x_k $, $k\in\N_0$ is stable and that $-1$ is not an eigenvalue of $\mAd$.
Then, the inverse Cayley transform
\begin{equation}\label{inv-cayley}
 \mAc :=(\mAd-\mI)(\mAd+\mI)^{-1}
\end{equation}
is well-defined and the following properties hold.

\begin{itemize}
 \item [(i)] 
If $\mAd$ is hypocontractive then $\mAc$ is negative hypocoercive.
 \item [(ii)] 
If $\mAd$ is semi-contractive 
then $\mAc$ is semi-dissipative.
Moreover, with $\mAH= \frac12 (\mAc+\mAc^\topH)$, the matrix $(\mAd+\mI)$ is a bijection from $\ker(\mI -\mAd^\topH \mAd)$ to $\ker(\mAH)$ and  $\dim\ker(\mAH) =\dim\ker(\mI -\mAd^\topH \mAd)$.
 %
\end{itemize}
\end{lemma}

\begin{proof}
Since $-1$ is not an eigenvalue of $\mAd$ then the matrices~$(\mAd+\mI)$, $(\mAd+\mI)^\topH$ are invertible; and the inverse Cayley transform~\eqref{inv-cayley} is well-defined.

(i) If $\mAd$ is hypocontractive then all eigenvalues of~$\mAd$ have modulus less  than one, hence, all eigenvalues of~$\mAc$ have negative real part.
Thus, $\mAc$ is negative hypocoercive.

(ii) 
If $\mAd\in\Cnn$ is semi-contractive then $x_{k+1} =\mAd x_k $, $k\in\N_0$ is stable (due to Proposition~\ref{prop:EV+SV}).
Then
\begin{equation} \label{mAH}
\begin{split}
 \mAH
 &=\tfrac12 (\mAc +\mAc^\topH)
\\
 &=\tfrac12 \big((\mAd-\mI)(\mAd+\mI)^{-1} +(\mAd+\mI)^{-\topH} (\mAd-\mI)^\topH\big)
\\
 &=\tfrac12 (\mAd+\mI)^{-\topH} \big((\mAd+\mI)^\topH (\mAd-\mI) +(\mAd-\mI)^\topH (\mAd+\mI)\big)(\mAd+\mI)^{-1}
\\
 &=-(\mAd+\mI)^{-\topH} (\mI -\mAd^\topH \mAd)(\mAd+\mI)^{-1} \,.
\end{split}
\end{equation}
Thus, the matrices~$-\mAH$ and~$(\mI -\mAd^\topH \mAd)$ are related via a congruence transformation, and hence $\mAd$ is semi-contractive (or equivalently, $(\mI -\mAd^\topH \mAd)$ is positive semi-definite) if $\mAc$ is semi-dissipative.

Due to~\eqref{mAH}, if $v\in\ker(\mAH)$ then $(\mAd+\mI)^{-1} v\in\ker(\mI -\mAd^\topH \mAd)$.
Thus, $(\mAd+\mI)^{-1}\ker(\mAH) \subseteq\ker(\mI -\mAd^\topH \mAd)$.
Conversely, if $w\in\ker(\mI -\mAd^\topH \mAd)$ then $(\mAd+\mI) w\in\ker(\mAH)$.
Thus, $(\mAd+\mI) \ker(\mI -\mAd^\topH \mAd)\subseteq \ker(\mAH)$.
Altogether, $(\mAd+\mI)$ is a bijection from $\ker(\mI -\mAd^\topH \mAd)$ to $\ker(\mAH)$ which implies that $\dim\ker(\mAH) =\dim\ker(\mI -\mAd^\topH \mAd)$.
\end{proof}

\begin{remark}\label{rem:other shift}
The assumption in Lemma~\ref{lemma:Cayley:Inverse} that $-1$ is not an eigenvalue of $\mAd$ can be relaxed by considering
$\mAc=(\mAd-\alpha \mI)(\mAd+\alpha \mI)^{-1}$, where $-\alpha\in \C$ (with $| \alpha |=1$) is not an eigenvalue of $\mAd$.
Such an $\alpha$ clearly exists in the complex case, but this will not work in the real case if both $1$ and $-1$ are eigenvalues of $\mAd$ and one wants to stay within the class of real matrices.
\end{remark}

The Cayley transformation also gives a direct relation between the hypocoercivity and hypo\-contractivity indices.
\begin{theorem} \label{thm:Cayley}
\begin{itemize}
 \item [(i)]
Let $\mAc\in\Cnn$ be semi-dissipative and negative hypocoercive and let $\mAd:=(\mI+\mAc)(\mI-\mAc)^{-1}$.
Then the hypocoercivity index~$\mHC\in\N_0$ of $\mAc$ and the hypocontractivity index~$\mDHC$ of $\mAd$ are the same, i.e., $\mDHC(\mAd)=\mHC(\mAc)$.
 \item [(ii)] 
Let $\mAd\in\Cnn$ be semi-contractive and hypocontractive and let $\mAc:=(\mAd-\mI)(\mAd+\mI)^{-1}$.
Then the hypocontractivity index~$\mDHC\in\N_0$ of $\mAd$ and the hypocoercivity index~$\mHC$ of $\mAc$ are the same, i.e., $\mHC(\mAc)=\mDHC(\mAd)$.
\end{itemize}
\end{theorem}

\begin{proof}
(i) Due to the assumptions and Lemma~\ref{lemma:Cayley}, $\mAd=2(\mI-\mAc)^{-1}-\mI$
is semi-contractive and hypocontractive.
Thus, by Lemma \ref{lemma:Cayley:Inverse}, the inverse Cayley transform~$(\mAd-\mI)(\mAd+\mI)^{-1}$ is well-defined and satisfies~$(\mAd-\mI)(\mAd+\mI)^{-1}=\mI-2(\mAd+\mI)^{-1} =\mAc$.

By assumption, the matrix~$\mAc=\mAH +\mAS$ has a finite HC-index~$\mHC=\mHC(\mAc)$ which is the smallest integer such that, due to \eqref{Tm:BS_BH2},
\[
 \bigcap_{j=0}^{\mHC} \ker \big(\mAH \mAc^j\big) =\{0\} .
\]
Hence, there exists a vector $v_0\in\Cn\setminus\{0\}$ such that
\begin{equation} \label{v0:AS}
 \mAc^j v_0\in \ker (\mAH) \ ,
 \qquad
 j\in\{0,\ldots,\mHC-1\}
 \qquad \text{and }
 \mAc^{\mHC} v_0 \notin \ker(\mAH) \ .
\end{equation}
Thus, by~Lemma~\ref{lemma:Cayley} (ii), we obtain that
\begin{equation} \label{Av0}
 (\mI-\mAc)\mAc^j v_0\in \ker (\mI -\mAd^\topH\mAd) \ ,
 j\in\{0,\ldots,\mHC-1\},
\ \text{and }
 (\mI-\mAc)\mAc^{\mHC} v_0 \notin \ker (\mI -\mAd^\topH\mAd) \ .
\end{equation}
Conversely, the existence of some $v_0\ne0$ satisfying the ``first part'' of \eqref{Av0} with some $\mHC\ge1$ implies that the HC-index of $\mAc$ is at least $\mHC$.

The matrix $\mAd=(\mI+\mAc)(\mI-\mAc)^{-1}$ is hypocontractive with HC-index~$\mDHC :=\mDHC(\mAd)\in\N_0$.
Due to \eqref{Dm:A} this is the smallest integer such that
\[
 \bigcap_{j=0}^{\mDHC} \ker \big((\mI-\mAd^\topH\mAd) \mAd^j\big) =\{0\} .
\]
Hence, there exists a vector $w_0\in\Cn\setminus\{0\}$ such that
\begin{equation} \label{mAd:kernel}
 w_0\in \bigcap_{j=0}^{\mDHC-1} \ker\big((\mI -\mAd^\topH\mAd) \mAd^j\big) \ ,
 \quad\text{ and }
 w_0\notin \ker\big((\mI -\mAd^\topH\mAd) \mAd^{\mDHC}\big), \ \end{equation}
or equivalently, there exists $w_0\in\Cn\setminus\{0\}$ such that
\begin{equation} \label{w0:mAd}
 \mAd^j w_0\in \ker (\mI -\mAd^\topH\mAd) \ ,
 \qquad
 j\in\{0,\ldots,\mDHC-1\}
 \qquad \text{and }
 \mAd^{\mDHC} w_0 \notin \ker(\mI -\mAd^\topH\mAd).
\end{equation}
Conversely, the existence of some $w_0\ne0$ satisfying the ``first part'' of \eqref{w0:mAd} with some $\mDHC\ge1$ implies that the dHC-index of $\mAd$ is at least $\mDHC$.

It remains to show that $\mHC(\mAc) =\mDHC(\mAd)$:
If $\mHC=0$, then $\mAc$ is dissipative such that $\ker(\mAH)=\{0\}$. Hence, $\ker(\mI-\mAd^\topH\mAd)=\{0\}$ due to~Lemma~\ref{lemma:Cayley} (ii) and $\mAd$ is contractive, i.e.\ $\mDHC=0$. Conversely, if $\mDHC(\mAd)=0$ then $\ker(\mI-\mAd^\topH\mAd)=\{0\}$. Hence, $\ker(\mAH)=\{0\}$ by~Lemma~\ref{lemma:Cayley} (ii) and thus $\mAc$ is dissipative and $\mHC =0$.

If $\mHC\geq 1$, then let $v_0 \in\Cn\setminus\{0\}$ satisfy~\eqref{Av0} with $\mHC=\mHC(\mAc)$. Hence,
$$
  q(\mAc)(\mI-\mAc)v_0 \in \ker(\mI-\mAd^\topH\mAd)
$$
for all polynomials $q$ of order up to $\mHC-1$. In particular
$$
  w_0:=(\mI-\mAc)^{\mHC} v_0\in \ker(\mI-\mAd^\topH\mAd),
$$
and $w_0\ne0$ since $(\mI-\mAc)$ is regular.
Also, using \eqref{cayley} we find that
$$
   \mAd^j w_0 = (\mI+\mAc)^j(\mI-\mAc)^{\mHC-j}v_0 \in \ker (\mI -\mAd^\topH\mAd) \ ,
 \qquad
 j\in\{0,\ldots,\mDHC-1\}\,.
$$
Hence, \eqref{w0:mAd} implies $\mDHC(\mAd)\geq \mHC(\mAc)$.

Conversely, if $\mDHC\geq 1$, then let $w_0\in\Cn\setminus\{0\}$ satisfy~\eqref{w0:mAd} with $\mDHC=\mDHC(\mAd)$. Hence,
$$
  q(\mAd)w_0 \in \ker(\mI-\mAd^\topH\mAd)
$$
for all polynomials $q$ of order up to $\mDHC-1$.
We define $v_0:=(\mAd+\mI)^{\mDHC} w_0\ne0$ since $(\mAd+\mI)$ is regular.
Using \eqref{inv-cayley} and $\mI-\mAc=2(\mAd+\mI)^{-1}$ we compute
\[
 \mAc^j (\mI-\mAc)v_0
 =2(\mAd-\mI)^j(\mAd+\mI)^{\mDHC-j-1}w_0 \in \ker (\mI -\mAd^\topH\mAd) \ ,
 \qquad
 j\in\{0,\ldots,\mDHC-1\}\,.
\]
Hence, \eqref{Av0} implies $\mHC(\mAc)\geq \mDHC(\mAd)$.
Altogether, we deduce that $\mDHC(\mAd) =\mHC(\mAc)$, which finishes the proof of statement~(i).

(ii) The proof is analogous to that of (i).
\end{proof}

\begin{remark}\label{rem:staffans}
It was pointed out to the authors that the results presented in Lemmas~\ref{lemma:Cayley} and~\ref{lemma:Cayley:Inverse} as well as Theorem~\ref{thm:Cayley} can be proved in an alternative way by using the characterization via unobservability subspaces, see  Remark~\ref{rem:unobservability}.
The results then can be proved via  Lemmas 12.3.10 and 12.2.6 of \cite{Sta05}.
\end{remark}

\begin{example}\label{ODE:HC2}
Consider the continuous-time system~\eqref{ODE:B} with the coefficient matrix
\begin{equation} \label{matrix:A:HC2}
 \mAc
 =\begin{bmatrix}
  0 & -1 & 0 & 0 \\
  1 & 0 & -1 & 0 \\
  0 & 1 & -1 & 0 \\
  0 & 0 & 0 & -1
 \end{bmatrix} \,
\end{equation}
which is semi-dissipative and $\mB=-\mAc$ 
has hypocoercivity index $\mHC=2$.
Applying the Cayley transformation gives
\begin{equation} \label{matrix:Ad:HC2}
 \mAd
 =(\mI+\mAc)(\mI-\mAc)^{-1}
 =\tfrac15 \begin{bmatrix}
  1 & -4 & 2 & 0 \\
  4 & -1 & -2 & 0 \\
  2 & 2 & -1 & 0 \\
  0 & 0 & 0 & 0
 \end{bmatrix} \,
\end{equation}
which is semi-contractive and has hypocontractivity index $\mDHC=2$.
\end{example}

Unfortunately, the Cayley transform does not relate the shifted hypocoercivity index~$\mSHC$ and the scaled hypocontractivity index~$\mDSHC$ in the same way, as the following example illustrates.
\begin{example}\label{ex:notshift}
Consider the matrix
\[
 \tAd =\begin{bmatrix} 0 & 1 & 0 \\ 0 & 0 & 1 \\ 0 & 0 & 0 \end{bmatrix}
\]
which is hypocontractive with hypocontractivity index $\mDHC=2$.
The matrix $\mAd :=2\tAd$ is not semi-contractive, since $\mAd^\topH\mAd=\diag(0,4,4)$, but it has scaled hypocontractivity index~$\mDSHC=2$.
For the inverse Cayley transform of $\mAd$ we obtain
\[
 \mAc
 :=
 (\mAd-\mI)(\mAd+\mI)^{-1}
 =
 \begin{bmatrix}
  -1 & 4 &-8 \\
   0 &-1 & 4 \\
   0 & 0 &-1
 \end{bmatrix}\,,\qquad
 \mAH
 =
 \begin{bmatrix}
  -1 & 2 &-4 \\
   2 &-1 & 2 \\
  -4 & 2 &-1
 \end{bmatrix}
 \,.
\]
The eigenvalues of $\mAH$ are
$\lambda=3$, $\lambda_\pm=-3\pm\sqrt{12}$ and hence they are simple and the shifted HC-index of $\mAc$ is $\mSHC=1$.
This example shows that $2=\mDSHC(\mAd)\ne\mSHC(\mAc)=1$.
\end{example}

It is well-known, see e.g.~\cite[page 180]{Ka80}, that the Cayley transformation also directly relates the stabilizing solutions of the discrete-time and continuous-time Lyapunov equation.
We summarize these results in the following Lemma.
\begin{lemma}\label{lem:lyapdc}
Let $\mAc\in\Cnn$ be a matrix such that~\eqref{ODE:B} is (Lyapunov) stable and let $\mAd=(\mI+\mAc)(\mI-\mAc)^{-1}$.
Then $\mP$ is the positive definite solution $\mP_c=\mP$ of the continuous-time Lyapunov equation
\[
 \mAc^\topH \mP_c +\mP_c\mAc =-\mQ_c,
\]
for some positive semidefinite matrix $\mQ_c$ if and only if $\mP$ is the positive definite solution $\mP_d=\mP$ of the discrete-time Lyapunov equation
\[
 \mAd^\topH \mP_d\mAd -\mP_d =-\mQd
\]
for positive semidefinite $\mQd$, where the right hand sides are related via
$\mQd=2(\mI-\mAc^\topH)^{-1}\mQ_c(\mI-\mAc)^{-1}$.
\end{lemma}

In summary, we have an almost complete analogy between the properties of continuous-time and discrete-time systems.
We summarize these invariance properties under the Cayley transformation and the inverse Cayley transformation (if it exists) in Table~\ref{invcayley}.
\begin{table}
\begin{center}
\begin{tabular}{ll}
 continuous-time & discrete-time, \\[1mm]
\hline\hline\\[-3mm]
 $\ddt x =\mA_c x$ for $t\geq 0$ & $x_{k+1} =\mAd x_k $ for $k\in\N_0$, \\
 \hline
 (asymptotically) stable & (asymptotically) stable, \\
 \hline
 semi-dissipative & semi-contractive,\\
 \hline
 (hypo)coercive & (hypo)contractive,\\
 \hline
 $\mHC(\mA_c)$ &  $\mDHC(\mAd)$,\\
 \hline
 Lyapunov solution $\mP_c$ &Lyapunov solution $\mP_d$.\\
 \hline\\
\end{tabular}
\caption{Invariance of properties of continuous-time and discrete-time systems under Cayley transformation $\mAd=(\mI+\mAc)(\mI-\mAc)^{-1}$ and inverse Cayley transformation $\mAc=(\mAd-\mI)(\mAd+\mI)^{-1}$ \label{invcayley}
}
\end{center}
\end{table}

Finally we consider the \emph{scaled Cayley transform}
\begin{equation}\label{sc-cayley}
  \mAd(t):=(\mI+\frac{t}{2}\mAc)(\mI-\frac{t}{2}\mAc)^{-1} \quad\mbox{for } t>0\ ,
\end{equation}
which can be considered as a short-time approximation of the matrix exponential for \eqref{ODE:B}. Due to the scaling invariance of the hypocoercivity (index) of a matrix $\mAc\in\Cnn$ (see \S\ref{sec:trafo}), we readily obtain:
\begin{corollary}\label{cor-sc-cayley}
Let $\mAc\in\Cnn$ be semi-dissipative and negative hypocoercive. Then, for all $t>0$, the scaled Cayley transform $\mAd(t)$ is hypocontractive (due to Lemma \ref{lemma:Cayley} (i)), its dHC-index satisfies $\mDHC(\mAd(t))=\mHC(\mAc)=:\mDHC$ (due to Theorem \ref{thm:Cayley} (i)), and the norm of its powers satisfy
$$
 \|\mAd(t)^j\|_2 =1 \text{ for all } j=1,\ldots,\mDHC\,,
 \quad\text{and }\
 \|\mAd(t)^{\mDHC+1}\|_2 <1 \,
$$
(due to Theorem \ref{DS:A:semi-contractive:power-bounds}).
\end{corollary}


\section*{Conclusions}
In this paper we have given a systematic analysis of different concepts related to the stability and short-time behavior of solutions to linear constant coefficient continuous-time and discrete-time systems.
While many results for the continuous-time setting were already established in \cite{AAM21} we have analyzed under which linear transformations the properties of asymptotic stability, semi-dissipativity and hypocoercivity stay invariant.

For linear time-invariant continuous-time systems, it is well-known that the exponential rate of the short-time behavior of the propagator norm~$\|e^{\mAc t}\|$ is determined by the logarithmic norm of the system matrix.
In this work, we established that the shifted hypocoercivity index characterizes the (remaining) algebraic decay of  the propagator norm in the short-time regime.

For each of the continuous-time results we have derived a corresponding result for the discrete-time case.
These include the relation between (asymptotic) stability, semi-contractivity and hypocontractivity.
We have also introduced the new concept of shifted hypocoercivity and scaled hypocontractivity.
We then have analyzed how the properties relate under the Cayley transformation that relates continuous-time and discrete-time systems.
While the role of the hypocontractivity index (or norm-one index) in the discrete-time setting has been recognized before, the corresponding concept---the hypocoercivity index---in the continuous-time setting and its role has been established only recently.

Future work will include the extension of the results of \cite{AAM21} for linear continuous-time dif\-fer\-en\-tial-algebraic systems to discrete-time descriptor systems.


{
\appendix
\section{Staircase forms}\label{sec:staircase}
In \cite{AAM21} a computationally feasible procedure to check the conditions of Lemma~\ref{lem:HC:equivalence} in the continuous-time case via a staircase form under unitary congruence transformations of the pair $(\mJ,\mR)=(\mBS,\mBH)$ has been presented.
\begin{lemma}[Staircase form for $(\mJ,\mR)$] \label{lem:SF}
Let $\mJ\in\Cnn$ be a skew-Hermitian matrix, and $\mR\in\Cnn$ be a nonzero Hermitian matrix.
Then there exists a unitary matrix $\mV\in\Cnn$, such that~$\mV \mJ \mV^\topH$ and~$\mV \mR \mV^\topH$ are block tridiagonal matrices of the form
\begin{equation} \label{matrices:staircase:J-R}
\begin{split}
\mV\ \mJ\ \mV^\topH
&= \begin{array}{l}
\left[ \begin{array}{ccccccc|c}
 \mJ_{1,1} & -\mJ_{2,1}^\topH & & & \cdots & & 0 & 0\\
 \mJ_{2,1} & \mJ_{2,2} & -\mJ_{3,2}^\topH & & & & &\\
  & \ddots & \ddots & \ddots & & & \vdots & \\
  & & \mJ_{k,k-1} & \mJ_{k,k} & -\mJ_{k+1,k}^\topH & & & \vdots \\
 \vdots & & & \ddots & \ddots & \ddots & & \\
  & & & & \mJ_{s-2,s-3} & \mJ_{s-2,s-2} & -\mJ_{s-1,s-2}^\topH & \\
 0 & & \cdots & & & \mJ_{s-1,s-2} & \mJ_{s-1,s-1} & 0\\ \hline
 0 & & & \cdots & & & 0 & \mJ_{ss}
\end{array}\right]
 \begin{array}{c}
  n_1\\ n_2\\ \vdots\\ \\ n_k\\ \vdots \\ n_{s-2} \\n_{s-1}\\ n_s
 \end{array} \\
 \quad\hspace{5pt} n_1 \hspace{170pt} n_{s-2} \hspace{30pt} n_{s-1}  \hspace{25pt} n_s
\end{array}, \quad \\
\mV\ \mR\ \mV^\topH
&= \begin{array}{l} \left[ \begin{array}{cc}
     \mR_1 & 0\\
     0 & 0 \\
     \vdots & \vdots\\
     \vdots & \vdots \\
     0 & 0 \end{array}\right]
 \begin{array}{c}
     n_1 \\ n_2 \\ \vdots\\ \vdots \\ n_s \end{array} \\
  \;\;\; n_1  \;\; n-n_1
\end{array},
\end{split}
\end{equation}
where $n_1 \geq n_2 \geq \cdots \geq n_{s-1} >0$, $n_s \geq 0$, and $\mR_1\in\C^{n_1,n_1}$ is nonsingular.

If $\mR$ is nonsingular, then~$s=2$ and~$n_2=0$.
For example, $\mV=\mI$, $\mJ_{1,1}=\mJ$ and~$\mR_1=\mR$ is an admissible choice.

If $\mR$ is singular, then $s\geq 3$ and the matrices~$\mJ_{i,i-1}$, $i=2,\ldots,s-1$, in the subdiagonal have full row rank and are of the form
\[
\mJ_{i,i-1} = \begin{bmatrix} \Sigma_{i,i-1} & 0 \end{bmatrix}, \quad
i =2,\ldots, s-1,
\]
%
 %
with nonsingular matrices~$\Sigma_{i,i-1}\in\C^{n_i,n_i}$, moreover $\Sigma_{s-1,s-2}$ is a real-valued diagonal matrix.
\end{lemma}

A system~\eqref{ODE:B} with an accretive matrix $\mB=\mBS+\mBH$ is hypocoercive if $n_s=0$ and if this is the case then the hypocoercivity index is $\mHC(\mB)=s-2$.

A similar staircase form can be derived in the discrete-time case.
It is based on the polar decomposition $\mA_d= \mU\mQ$, see Proposition \ref{prop:PolDec}. 
%
\begin{lemma}[Staircase form for $(\mU,\mQ)$] \label{lem:SFQU}
Let $\mU\in\Cnn$ be a unitary matrix, and $\mQ\in\Cnn$ be a nonzero semi-contractive Hermitian matrix.
Then there exists a unitary matrix $\mV\in\Cnn$, such that~$\mV \mQ \mV^\topH$ and~$\mV \mU \mV^\topH$ are block upper Hessenberg matrices of the form
\begin{equation} \label{matrices:staircase:HU}
\begin{split}
\mV\ \mU\ \mV^\topH
&= \begin{array}{l}
\left[ \begin{array}{ccccc|c}
 \mU_{1,1} & \mU_{1,2} &\cdots & \cdots &  \mU_{1,s-1} & 0\\
 \mU_{2,1} & \mU_{2,2} & \mU_{2,3} & \cdots &  \mU_{2,s-1} &0\\
  & \ddots & \ddots & \ddots &  \ddots & \vdots  \\
   & & \mU_{s-2,s-3} & \mU_{s-2,s-2} & \mU_{s-2,s-1} & 0\\
 0 & \cdots &  0& \mU_{s-1,s-2} & \mU_{s-1,s-1} & 0\\ \hline
 0 &  \cdots & & & 0 & \mU_{s,s}
\end{array}\right]
 \begin{array}{c}
  n_1\\ n_2\\ \vdots \\ n_{s-2} \\n_{s-1}\\ n_s
 \end{array} \\
%
%
\end{array}, \quad \\
\mV\ \mQ\ \mV^\topH
&= \begin{array}{l} \left[ \begin{array}{ccccc|c}
     \mQ_1 & 0 & \cdots & \cdots &0 &0\\
     0 &  \mI_{n_2} & 0 & \cdots & \vdots & \vdots\\
     \vdots &  0 & \ddots & \ddots & \vdots & \vdots\\
         \vdots &  \ddots  & \ddots & \ddots & \vdots & \vdots\\
     \vdots & \vdots  & \ddots  & 0 & \mI_{n_{s-1}} & 0\\
     \hline
     0 & 0 & \cdots & \cdots & 0 & \mI_{n_{s}} \end{array}\right]
 \begin{array}{c}
     n_1 \\ n_2 \\ \vdots\\ \\ \vdots \\ n_{s-1}
      \\ n_s \end{array} \\
\end{array},
\end{split}
\end{equation}
where $n_1 \geq n_2 \geq \cdots \geq n_{s-1}>0$, $n_s \geq 0$, and $\mQ_1\in\C^{n_1,n_1}$ is contractive and Hermitian.

If $\mQ$ is contractive, then~$s=2$ and~$n_2=0$.
Then $\mV=\mI$, $\mU_{1,1}=\mU$ and~$\mQ_1=\mQ$ is an admissible choice.

If $\mQ$ is not contractive, then $s\geq 3$ and the matrices~$\mU_{i,i-1}$, $i=2,\ldots,s-1$, in the subdiagonal have full row rank and are of the form
\[
 \mU_{i,i-1}
 =\begin{bmatrix} \Sigma_{i,i-1} & 0 \end{bmatrix}, \quad i =2,\ldots, s-1,
\]
with nonsingular matrices~$\Sigma_{i,i-1}\in\C^{n_i,n_i}$, moreover $\Sigma_{s-1,s-2}$ is a real-valued diagonal matrix.
\end{lemma}
\begin{proof}
If~$\mQ$ is contractive, then~$n_1=n$ and we have to choose~$s=2$ and~$n_2=0$ to fit~$\mU$ into the proposed structure in~\eqref{matrices:staircase:HU}.
\\
If~$\mQ$ is not contractive, then we have the following constructive proof.\\
\begin{breakablealgorithm}
\caption{Staircase algorithm for pair~$(\mU,\mQ)$}
\label{algorithm:staircase:U-H}
\begin{algorithmic}[1]
\REQUIRE $(\mU,\mQ)$\\
 ----------- {\em Step 0} -----------
\STATE Perform a (spectral) decomposition of $\mQ$ such that
\[
\mQ
=\mV_1 \begin{bmatrix} \tQ_1 & 0 \\ 0 & \mI \end{bmatrix} \mV_1^\topH ,
\]
with $\mV_1\in \Cnn$ unitary, $\tQ_1\in \C^{n_1, n_1}$ contractive and Hermitian.
\STATE Set $\mV := \mV_1^\topH$, $\tQ:=\mV_1^\topH\ \mQ\ \mV_1$,
\[
 \tU
 := \mV_1^\topH\ \mU\ \mV_1
 =:\begin{bmatrix}
   \tU_{1,1} & \tU_{1,2} \\
   \tU_{2,1} & \tU_{2,2}
   \end{bmatrix}.
\]
%
%
\newline ----------- {\em Step 1} -----------
\STATE Perform a singular value decomposition (SVD) of~$\tU_{2,1}\in\C^{(n-n_1)\times n_1}$ such that
\[
 \tU_{2,1}
 = \mW_{2,1} \begin{bmatrix} \tSigma_{2,1} & 0\\ 0 & 0 \end{bmatrix} \mV^\topH_{2,1},
\]
with unitary matrices~$\mW_{2,1}$ and~$\mV_{2,1}$ as well as a positive definite, diagonal matrix $\tSigma_{2,1}\in \mathbb R^{n_2, n_2}$.
\STATE Set $\mV_2 := \diag(\mV_{2,1}^\topH,\ \mW_{2,1}^\topH)$,\ $\mV:= \mV_2 \mV$.
%
\STATE Set
\[
\def\arraystretch{1.4}
\tU := \mV_2\ \tU\ \mV_2^\topH
=: \left[ \begin{array}{c|cc}
 \tU_{1,1} & \tU_{1,2} & \tU_{1,3}\\
 \hline
 \tU_{2,1} & \tU_{2,2} & \tU_{2,3} \\
 0     & \tU_{3,2} & \tU_{3,3}
\end{array}\right],\qquad
\def\arraystretch{1}
\tR
:= \mV_2 \tQ \mV_2^\topH
=: \left[ \begin{array}{c|cc}
\tQ_1 & 0 & 0 \\
\hline
0 & \mI_{n_2} & 0 \\
0 & 0 & \mI
\end{array}\right].
\]
(The lines indicate the partitioning of the block matrices~$\tU$ and~$\tQ$ in the previous step.)
\newline ----------- {\em Step 2} -----------
\STATE $i := 3$
\WHILE{$n_{i-1} > 0$ \OR $\tU_{i,i-1} \neq 0$}
\STATE Perform an SVD of $\tU_{i,i-1}$ such that
\[
\tU_{i,i-1}
= \mW_{i,i-1} \begin{bmatrix} \tSigma_{i,i-1} & 0\\ 0 & 0 \end{bmatrix} \mV^\topH_{i,i-1} ,
\]
with unitary matrices~$\mW_{i,i-1}$ and~$\mV_{i,i-1}$ as well as a positive definite, diagonal matrix $\tSigma_{i,i-1}\in \R^{n_i, n_i}$.
\STATE Set $\mV_i := \diag(\mI_{n_1},\ldots,\mI_{n_{i-2}},\ \mV_{i,i-1}^\topH,\ \mW_{i,i-1}^\topH )$, $\mV:= \mV_i \mV$.
%
\STATE Set
\[
\tU
:= \mV_i\ \tU\ \mV_i^\topH
=: \begin{bmatrix}
 \tU_{1,1} & \tU_{1,2}  & \cdots      & \cdots      & \tU_{1,i+1} \\
 \tU_{2,1} & \tU_{2,2}  & \tU_{2,3}   &             & \vdots \\
    0      & \ddots     & \ddots      & \ddots      & \\
 \vdots    & \ddots     & \tU_{i,i-1} & \tU_{i,i}   & \tU_{i,i+1} \\
    0      & \cdots     &     0       & \tU_{i+1,i} & \tU_{i+1,i+1}
\end{bmatrix} ,
\
\text{where } \tU_{i,i-1} = [ \tSigma_{i,i-1}\quad 0].
\]
\STATE $i:= i+1$
\ENDWHILE
\newline ----------- {\em Step 3} -----------
\STATE $s:= i$
\FOR{$i=1,\ldots,s$}
 \FOR{$j=i,\ldots,s$}
  \STATE Set $\mU_{i,j} :=\tU_{i,j}$.
 \ENDFOR
\ENDFOR
\FOR{$i=2,\ldots,s$}
  \STATE Set $\mU_{i,i-1} :=\tU_{i,i-1}$.
\ENDFOR
\ENSURE Unitary matrix~$\mV$. 
\\
\end{algorithmic}
\end{breakablealgorithm}

\bigskip

It is clear that Algorithm~\ref{algorithm:staircase:U-H} terminates after a finite number of steps, either with~$n_{i-1}=0$ or~$\mU_{i,i-1}=0$.
We also note that Step 3 provides the nonzero entries of the r.h.s.\ of $\mV\mU\mV^\topH$ in \eqref{matrices:staircase:HU}.
\end{proof}
%

Note that Algorithm~\ref{algorithm:staircase:U-H} can be applied to the polar decomposition $\mPd\mUd$ 
analogously.
In both cases it immediately follows that $\mA_d$ is hypocontractive if $n_s=0$ and the hypocontractivity index is then $\mDHC(\mAd)=s-2$.

\section{Equivalent hypocoercivity conditions}

The following lemma is a simple generalization of Lemma 2.3 in \cite{AASt15} and Proposition 1 in \cite{AAC18}.
\begin{lemma}\label{lem:Definiteness}
Let $\mD\in\Cnn$ be positive semi-definite and $\mC\in\C^{n\times n}$.
Then the following are equivalent:
\begin{itemize}
\item [(E1)]
There exists $m\in\N_0$ such that
\begin{equation}\label{condition:KalmanRank-App}
 \rank[{\mD},\mC{\mD},\ldots,\mC^m {\mD}]=n \,.
\end{equation}
\item [(E2)] 
There exists $m\in\N_0$ such that
\begin{equation}\label{index-App}
 \sum_{j=0}^m \mC^j \mD (\mC^\topH)^j > 0 \,.
\end{equation}
\end{itemize}
Moreover, the smallest possible $m\in\N_0$ in (E1) and (E2) coincide.
\end{lemma}
\begin{proof}
First, we show that (E1) is equivalent to:
\begin{itemize}
\item [(E1')]
There exists $m\in\N_0$ such that
\begin{equation*}
 \rank[{\mD^{1/2}},\mC{\mD}^{1/2},\ldots,\mC^m {\mD}^{1/2}]=n \,,
\end{equation*}
\end{itemize}
with the same $m$ as in (E1):\\
(E1) holds iff the statement
\[
 x^\topH [{\mD},\mC{\mD},\ldots,\mC^m {\mD}]=0\quad \mbox{for some } x\in\C^n,
\]
i.e.\ $\mD(\mC^\topH)^j x=0$ for $j=0,\ldots,m$ implies $x=0$. Now, since $\ker (\mD)= \ker (\mD^{1/2})$, (E1) and (E1') are equivalent.

Next, let (E1) hold and define
\[
 \mE:=[\mD^{1/2},\,\mC\mD^{1/2},...,\,\mC^m\mD^{1/2}] \in\C^{n\times(m+1)n}\,.
\]
Then,
\[
 \C^{n\times n}\ni \mE\ \mE^\topH = \sum_{j=0}^{m} \mC^j\mD(\mC^\topH)^j \ge 0
\]
has rank $n$ and \eqref{index-App} follows.\\
Conversely, let (E2) hold but assume we had $\rank \mE<n$.
Then, $\exists \,0\ne x\in\C^n$ with $x^\topH\mE=0$.
Hence, $x^\topH\mE\ \mE^\topH=0$ would contradict \eqref{index-App}.
\end{proof}

}

\section*{Acknowledgments}

The first author (FA) was supported by the Austrian Science Fund (FWF) via the FWF-funded SFB \# F65.
The second author (AA) was supported by the Austrian Science Fund (FWF), partially via the FWF-doctoral school "Dissipation and dispersion in non-linear partial differential equations'' (\# W1245) and the FWF-funded SFB \# F65.
The third author (VM) was supported by Deutsche Forschungsgemeinschaft (DFG) via the DFG-funded SFB \# 910.

%

\bibliography{AAM-2022-HCIndex}
\end{document}